\documentclass[11pt,oneside]{amsart}
\usepackage[T1]{fontenc}
\usepackage[english]{babel}
\usepackage{fouriernc} 
\usepackage{mathtools}
\usepackage{caption}
\usepackage{subcaption}
\usepackage{ stmaryrd }
\usepackage{extarrows}

\usepackage{thmtools, thm-restate}

\usepackage{adjustbox}

\usepackage{geometry} 
\usepackage{tikz} 
\usetikzlibrary {arrows.meta,bending,positioning}
\usetikzlibrary{math} 
\usepackage
{hyperref}
\hypersetup{linktoc=all,} 
\usepackage{hyperref} 

\usepackage{amsmath,amsthm,amssymb}
\usepackage[all,cmtip]{xy}
\usepackage{caption}
\usepackage{subcaption}

\usepackage[capitalise]{cleveref}  
\newcommand{\clevertheorem}[3]{%
  \newtheorem{#1}[thm]{#2}
  \crefname{#1}{#2}{#3}
}

\numberwithin{equation}{section} 
\numberwithin{figure}{section} 

\theoremstyle{plain} 
\newtheorem{thm}{Theorem}[section]
\crefname{thm}{Theorem}{Theorems}
\newtheorem*{thm*}{Theorem}
\clevertheorem{prop}{Proposition}{Propositions}
\newtheorem*{prop*}{Proposition}
\clevertheorem{lemma}{Lemma}{Lemmas}
\clevertheorem{cor}{Corollary}{Corollaries}
\clevertheorem{conj}{Conjecture}{Conjectures}

\theoremstyle{definition} 
\clevertheorem{definition}{Definition}{Definitions}
\clevertheorem{notn}{Notation}{Notations}

\theoremstyle{remark} 
\clevertheorem{remark}{Remark}{Remarks}
\newtheorem*{remark*}{Remark}
\clevertheorem{example}{Example}{Examples}
\clevertheorem{conv}{Assumption}{Assmputions}

\makeatletter\let\c@equation\c@thm\makeatother
\makeatletter\let\c@figure\c@thm\makeatother

\crefname{figure}{Figure}{Figures}

\crefname{equation}{Display}{Displays} 
\crefname{eq}{Display}{Displays}
\crefname{eqn}{Display}{Displays}

\hypersetup{
 pdfkeywords={latex starter,template},
 pdfauthor={Niles Johnson},
}

\usepackage{quiver}
\usepackage{comment}
\usepackage[normalem]{ulem}
\usepackage{graphicx,
tikz,tikz-cd} 
\usepackage{soul}

\usepackage{braket}

\newcommand{\shspace}{\hspace{.5mm}} 

\newcommand{\sh}[1]{{\ensuremath{\shspace\hspace{1mm}\makebox[-1mm]{$\langle$}\makebox[0mm]{$\langle$}\hspace{1mm}{#1}\makebox[1mm]{$\rangle$}\makebox[0mm]{$\rangle$}\shspace}}}

\usepackage{todonotes}

\DeclareMathOperator{\End}{End}
\DeclareMathOperator{\Mor}{Mor}
\newcommand{\xto}[1]{\xrightarrow{#1}}

\DeclareMathOperator{\tr}{tr}
\DeclareMathOperator{\Map}{Map}
\DeclareMathOperator{\id}{id}
\DeclareMathOperator{\Mod}{Mod}

\DeclareMathOperator{\Hom}{Hom}

\DeclareMathOperator{\ob}{ob}

\newcommand{\ghost}{iterated trace }

\begin{document}

\title{Frobenius and Verschiebung for $K$-theory of endomorphisms}
\author{Sanjana Agarwal}
\address{Indiana University}
\email{sanjagar@iu.edu}
\author{Jonathan Campbell}
\address{Center for Communications Research}
\email{jonalfcam@gmail.com}
\author{Diego Manco}
\address{University of Western Ontario}
\email{dmanco@uwo.ca}
\author{Kate Ponto}
\address{University of Kentucky}
\email{kate.ponto@uky.edu}
\author{Zhonghui Sun}
\address{Michigan State University}
\email{sunzhon1@msu.edu}
\date{\today}

\begin{abstract}
We show that the Frobenius and Verschiebung maps that are fundamental to Witt vectors lift to the reduced $K$-theory of endomorphisms.  In particular, we define  Frobenius and Verschiebung maps  for the reduced $K$-theory of twisted endomorphisms of modules over non commutative rings and show they have the expected behavior after applying the \ghost map. 
\end{abstract}

\maketitle

\setcounter{tocdepth}{1}
\tableofcontents

\section{Introduction}
The Witt vectors have wide applications in mathematics and a rich  structure that can, unfortunately, feel opaque.  The ring structure is a good example of this -- standard definitions \cite{Witt1937} 
are unmotivated and byzantine.  In \cite{ALM1} Almkvist shows that the structure of the Witt vectors is essentially a consequence of linear algebra, in a way made precise below. This paper is concerned with an exploration of this fact for non-commutative rings.

Stated somewhat grandiosely, linear algebra is the study of additive invariants on the category of modules over a commutative ring. Algebraic $K$-theory, by its very definition, represents a universal source for such invariants. Thus, all invariants of interest to linear algebra over a commutative ring $A$ have representatives in the group $\widetilde{K}_0 (\End (A))$,  whose exact definition\footnote{pun intended} we defer until the next section.

Linear algebra is not solely concerned with additive structure; the category of modules has a tensor product operation which induces a ring structure on $\widetilde{K}_0(\End (A))$. One would expect the multiplicative structure to be somewhat complicated since additive invariants and multiplicative structures seldom interact well. The key fact that Almkvist shows is that the friendly and familiar characteristic polynomial produces a ring homomorphism
\begin{equation}\label{eq:char}\mathrm{ch}\colon\widetilde{K}_0(\End (A))\to W(A).\end{equation}

Even better, $W(A)$ carries a topology and Almkvist shows that the image of the characteristic polynomial is dense in $W(A)$. Thus, identities on the dense image in $W(A)$ can first be proved in $\widetilde{K}_0 (\End (A))$, transported across the characteristic polynomial, and then shown to hold everywhere by density. In fact, philosophically, we view the $K$-theory of endomorphisms as the more fundamental object, with the Witt vectors being some sort of completion. One could then hope that all structure on the Witt vectors is actually a reflection of operations performed on the category $\End(A)$.  Indeed, Almkvist  \cite[p. 319]{ALM1} defines lifts of the Frobenius and Verschiebung maps \eqref{eq:intro:FV} to $\widetilde{K}_0(\End (A))$ using only linear algebra operations and verifies standard identities. 

With the view that the $K$-theory of endomorphisms is the primordial object, we can hope to understand the structure present in non-commutative Witt vectors, even if one does not want to define such an object directly. A careful study of additive invariants of the category of endomorphisms of modules over general rings was initiated in \cite{ponto} and \cite{Ponto_2012}, and an appropriate notion of trace was defined therein. These traces provide a substitute for the characteristic polynomial.  In this paper, we illustrate structures on the $K$-theory of endomorphisms. We define the Frobenius and Verschiebung maps, and these definitions are \emph{direct, elementary, and easily motivated}.  They are also simple enough that we can show they satisfy all the expected properties by \emph{direct computation}.

We define the Frobenius and Verschiebung maps in the twisted endomorphism category.
\begin{example}
\label{motivating example}
    For (not necessarily commutative) rings $R$ and $S$,
    $\Mod_{R,S}^c$ is the category whose objects are $R$-$S$-bimodules that are finitely generated and projective as $S$-modules.  Morphisms are bimodule homomorphisms.

    Let $M$ be an $R$-$R$-bimodule and $N$ be an $S$-$S$-bimodule. The objects of the {\bf twisted endomorphism category} $\End(R,S;M,N)$ are $R$-$S$-bimodule homomorphisms
\[f\colon  M\otimes_R P\to P\otimes_S N\]
where $P\in \Mod_{R,S}^c$.   
Morphisms $(f,P) \to (g,Q)$ are homomorphisms $\phi\colon P\to Q$ so that 
\[\begin{tikzcd}
	{M\otimes_RP} & {P\otimes_S N} \\
	{M\otimes_RQ} & {Q\otimes_S N}
	\arrow["f", from=1-1, to=1-2]
	\arrow["{\mathrm{id}\otimes \phi}"', from=1-1, to=2-1]
	\arrow["{\phi\otimes \mathrm{id}}", from=1-2, to=2-2]
	\arrow["g", from=2-1, to=2-2]
\end{tikzcd}\]
commutes. 
\end{example}

Let \(\widetilde{K}_0(R,S;M,N)\) be the reduced zeroth \(K\)-theory of \(\End(R,S;M,N)\), as defined in \cref{def:reduced_k_theory}, following the construction of \cite[p.~291]{ALM1}. Then our main results parallel those  in \cite{DKNP} for the Witt vectors.

\begin{thm}[\cref{def:frobenius_alt,prop:iterate_frobenius,thm:frobeius_ghost}]
Let  $R$ and $S$ be rings, $M$ be a flat $R$-$R$-bimodule, $N$ be an $S$-$S$-bimodule, and $n$ be a natural number.  There is a homomorphism 
\[F^n\colon \tilde{K}_0(R,S;M,N)\to \tilde{K}_0(R,S;M^{\otimes_R n},N^{\otimes_S n})\]
so that 
$F^nF^m=F^{nm}$ and  the image of $F^n(\phi)$ under the \ghost map (\cref{def:ghost})
is 
\[(\tr(F^n(\phi)), \tr(F^{2n}(\phi)), \tr(F^{3n}(\phi)), \ldots)\]
where $\tr$ is the Hattori-Stallings trace (\cref{defn:bicat_trace}).\end{thm}

\begin{thm}[\cref{defn:verschiebung_alt,prop:vnvm,thm:v_ghost}] Let  $R$ and $S$ be rings, $M$ be a flat $R$-$R$-bimodule, $N$ be an $S$-$S$-bimodule, and $n$ be a natural number.  If either $M=R$ or $N=S$, there is a homomorphism 
\[V^n\colon \tilde{K}_0(R,S;M^{\otimes_R n},N^{\otimes_S n})\to \tilde{K}_0(R,S;M,N)\]
so that 
$V^nV^m=V^{nm}$ and  the image of $V^n(\phi)$ under the \ghost map 
is 
\[(\underbrace{0,\ldots, 0}_{n-1}, \tau(\tr(F^1(\phi))), \underbrace{0,\ldots, 0}_{n-1},\tau(\tr(F^{2}(\phi))), \underbrace{0,\ldots, 0}_{n-1},\ldots)),\]
where $\tau$ is the transfer (\cref{def:transfer}) taken with respect to the $\mathbb{Z}/n$ action in \eqref{Action_on_trace}.
\end{thm} 

We emphasize that the proofs of these theorems depend heavily on beautiful structural properties of algebraic $K$-theory. In particular, the reduced algebraic $K$-theory of endomorphisms is a shadow in the sense of \cite{ponto}. This fact was also observed in \cite{gepner} where a proof sketch is given. In Appendix A we give a full proof in our more quotidian setting so that the paper is reasonably self-contained. Once $K$-theory is realized as a shadow, it inherits a great deal of structure, and visibly cyclic structures, which are important to our proofs.

We draw significant inspiration (though no direct results) from previous work on the Witt vectors. The papers \cite{ALM1,
GRAYSON1,GRAYSON2,DRESS19891} are primary among them. Notions of equivariance have shown up before in papers on Witt vectors, notably in \cite{DRESS198887}. 
There, the authors generalize the work of Witt and Cartier to define $G$-typical Witt vectors for any profinite group $G$.  Hesselholt  \cite{Hesselholt1997WittVO} defines $p$-typical Witt vectors for non-commutative rings and Krause and Nikolaus \cite{Lectures} define big Witt vectors for non-commutative rings.  

Most directly relevant to this paper,  in \cite{DKNP} the authors define big Witt vectors $W(A;M)$ for a non-commutative ring $A$ with coefficients in an $A$-bimodule $M$. In this setting,  $W(A;M)$ is an abelian group, with ring multiplication replaced by an external product 
\[W(R;M)\otimes W(S;N)\to W(R\otimes S; M\otimes N).\] Further, there are Frobenius and Verschiebung maps 
\begin{equation}\label{eq:intro:FV}F^n: W(R;M) \to W(R;M^{\otimes n}), \ \ \ V^n: W(R;M^{\otimes n}) \to W(R;M)\end{equation}
satisfying identities (\cref{structure on End}\eqref{structure on End_id}), whose images under the ghost map generalize the classical case.

The current paper takes a somewhat orthogonal approach to Witt vector constructions. One could either view Witt vectors as a fundamental object or the $K$-theory of endomorphisms as a fundamental object, with Witt vectors appearing through some completion process. In this paper, we take the latter view, and thus study the structure of the $K$-theory of endomorphisms for its own sake.

\subsection*{Outline}
In \cref{sec:com_ring} we recall the construction of the Frobenius and Verschiebung operations on $\tilde{K}_0(\End(A))$ for a commutative ring $A$. In \cref{sec:K-end} we define an endomorphism category generalizing \cref{motivating example} and an isomorphism essential to our results.  
The proofs in this section use techniques different from  those in  the rest of the paper, so we postpone them to \cref{appendix}.  In \cref{sec:new_maps_frobenius} we define the Frobenius map on the reduced $K$-theory of endomorphisms, and \cref{sec:new_maps_verscheibung} is the corresponding section for the Verschiebung map. 
\cref{sec:define_ghost} computes the result of applying the \ghost map to the Frobenius and Verschiebung.  In \cref{ap:add} we give an explicit proof of additivity of the trace in the context of \cref{motivating example}.  This is used in  \cref{sec:define_ghost} to show the trace descends to reduced $K$-theory.

\subsection*{Acknowledgments}
This project was part of the Collaborative Workshop held at Vanderbilt University in Summer 2024 and supported by the NSF Focused Research Group \emph{Collaborative Research: Trace Methods and Applications for Cut-and-Paste K-Theory}.  The authors were partially supported by NSF grants 
DMS-2052905 
and 
DMS-2052977. 
Jonathan Campbell thanks the Center for Communications Research -- La Jolla for its support.

\section{The case of a commutative ring}\label{sec:com_ring}

In this section, we recall the constructions in \cite{ALM3,facets}  of 
the Frobenius and Verschiebung maps on $\widetilde{K}_0(\End(A))$
for a commutative ring $A$, which induce the corresponding operations on $W(A)$.

\begin{example}
\label{classical_End(A)}
    For a commutative ring $A$, let $\Mod_A^c$ denote the category whose objects are finitely generated projective $A$-modules and morphisms are $A$-module maps.
 Let $\End(A)$ 
be the special case of \cref{motivating example} where $A=R=S=M=N$.
\end{example}

\begin{thm}\label{structure on End}\cite[p. 319]{ALM1}
    The category $\End(A)$ has the following structure:
\begin{enumerate}
    \item Two monoidal structures, one given by direct sum of objects and morphisms and the other given by tensor product. 
    \item\label{structure on End_F} Frobenius functors for each natural number $i$,
    \begin{align*}
    F^i: \End(A) & \to \End(A)\\
    (f,P) & \mapsto (f^{\circ i},P).
\end{align*}

\item\label{structure on End_V} Verschiebung functors for each natural number $i$,
    \begin{align*}
    V^i: \End(A) & \to \End(A)\\
    (f,P) & \mapsto (V^i(f),P^{\oplus i})
    \end{align*}
    where $V^i(f)$ is the composite 
$$
(p_0,p_1,\ldots, p_{n-1}) \xmapsto{t} (p_1, p_2, \ldots, p_{n-1}, p_0 )\xmapsto{f\oplus 1\oplus 1\cdots\oplus 1} (f(p_1), p_2, \ldots, p_{n-1}, p_0 ).$$

\item\label{structure on End_id} Identities that relate the Frobenius and Verschiebung including $$F^nF^m=F^{nm}, \ V^nV^m=V^{nm}, \ \text{and} \ F^nV^n = (-)^{\oplus n}.$$
\end{enumerate}
\end{thm}

The identities in \cref{structure on End}\eqref{structure on End_id} are suggested by the pictures for the Frobenius and Verschiebung in \cref{fig:classical}.  
(We include the intermediate step in \cref{fig:classical_frob} since the symmetry made explicit there  will be important in \S\ref{sec:new_maps_frobenius}.)  See \cref{fig:overall:v3v2} for an example of $V^nV^m=V^{nm}$ and \cref{eg:fnvn} for $F^nV^n=(-)^{\oplus n}$.

\begin{figure}
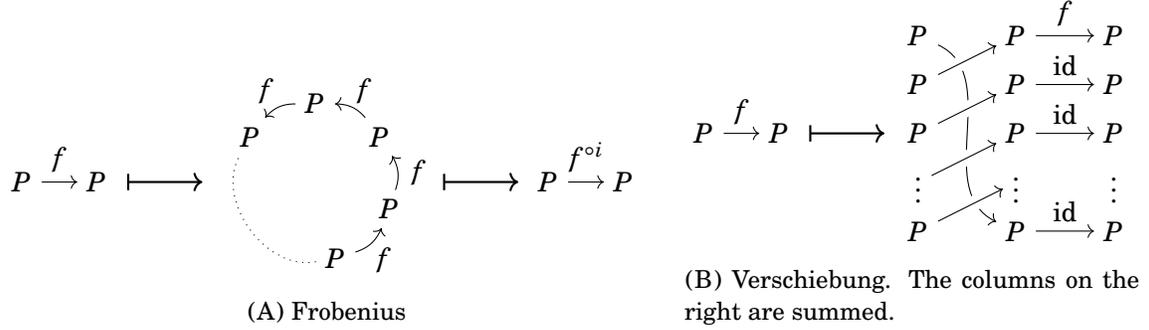

     \begin{subfigure}[b]{0.57\textwidth}
\tikz{
\tikzmath{\x1=55;}
\tikzmath{\x2=.35;}

\node (fc0) at (90:3*\x2) {$P$};
\node (fc1) at (90-1*\x1:3*\x2) {$P$};
\node (fc2) at (90-2*\x1:3*\x2) {$P$};
\node (fc3) at (90-3*\x1:3*\x2) {$P$};
\node (fcn1) at ({90+\x1}:3*\x2) {$P$};

\draw [<-] (fc0) edge [bend left=25] node[midway, above right=-3pt] {$f$} (fc1);
\draw [<-] (fc1) edge [bend left=25] node[midway, right=1pt] {$f$} (fc2);
\draw [<-] (fc2) edge [bend left=25] node[midway, below right=-3pt] {$f$} (fc3);
\draw [<-] (fcn1) edge [bend left=25] node[midway, above left=-3pt] {$f$} (fc0);
\draw [dotted] (fc3) edge [bend left=65] (fcn1);

\node (d1) at (-3.9,0){$P$};
\node (d2) at (-2.9,0){$P$};
\draw [->](d1) edge node [midway, above]{$f$} (d2);

\draw [|->, thick ] (-2.5,0) -- 
(-1.5,0);

\node (t1) at (3.1,0){$P$};
\node (t2) at (4.1,0){$P$};
\draw [->](t1) edge node [midway, above]{$f^{\circ i}$} (t2);

\draw [|->, thick] (1.7,0) -- 
(2.7,0);
}
\caption{Frobenius}
\label{fig:classical_frob}
\end{subfigure}
\hfill
     \begin{subfigure}[b]{0.4\textwidth}
\tikz{
\tikzmath{\x1=.65;}
\tikzmath{\x2=.65;}

\node (fc0) at (-1*\x1, 2*\x2) {$P$};
\node (fc1) at (-1*\x1,1*\x2) {$P$};
\node (fc2) at (-1*\x1, 0*\x2) {$P$};
\node (fc3) at (-1*\x1, -1*\x2) {$\vdots$};
\node (fcn1) at (-1*\x1, -2*\x2) {$P$};

\node (gc0) at (1*\x1, 2*\x2) {$P$};
\node (gc1) at (1*\x1,1*\x2) {$P$};
\node (gc2) at (1*\x1, 0*\x2) {$P$};
\node (gc3) at (1*\x1, -1*\x2) {$\vdots$};
\node (gcn1) at (1*\x1, -2*\x2) {$P$};

\node (hc0) at (3*\x1, 2*\x2) {$P$};
\node (hc1) at (3*\x1,1*\x2) {$P$};
\node (hc2) at (3*\x1, 0*\x2) {$P$};
\node (hc3) at (3*\x1, -1*\x2) {$\vdots$};
\node (hcn1) at (3*\x1, -2*\x2) {$P$};

\draw [->] (fc0)
edge [out=-25,in=155]  node[midway, above] {} (gcn1) ;

\draw [-, line width=2mm, white] (fc1)
edge  (gc0) ;

\draw [->] (fc1)
edge  
(gc0) ;

\draw [->] (gc0)
edge  node[midway, above] {$f$} (hc0) ;

\draw [-, line width=2mm,white] (fc2)
edge  node[midway, above ] {} (gc1) ;

\draw [->] (fc2)
edge  
(gc1) ;

\draw [->] (gc1)
edge  node[midway, above] {$\id$} (hc1) ;

\draw [-, line width=2mm, white] (fc3)
edge  (gc2) ;

\draw [->] (fc3)
edge 
(gc2) ;

\draw [->] (gc2)
edge  node[midway, above] {$\id$} (hc2) ;

\draw [-, line width=2mm, white] (fcn1)
edge (gc3) ;

\draw [->] (fcn1)
edge 
(gc3) ;

\draw [->] (gcn1)
edge  node[midway, above] {$\id$} (hcn1) ;

\node (d1) at (-3.5,0){$P$};
\node (d2) at (-2.5,0){$P$};
\draw [->](d1) edge node [midway, above]{$f$} (d2);

\draw [|->, thick ] (-2.1,0) -- (-1.1,0);

}
\caption{Verschiebung.  The columns on the right are summed.}\label{fig:classical_versh}
\end{subfigure}
\caption{Functors in $\End(A)$}\label{fig:classical}
\end{figure}

A sequence 
\[(f',P')\xrightarrow{\phi}(f,P)\xrightarrow{\psi} (f'',P'')\] in 
$\End(A)$ is {\bf exact} if
$P'\xrightarrow{\phi}P\xrightarrow{\psi}P''$
is exact in $\Mod_A^c$. 
The map 
\begin{align*}
    \Mod_A^c & \to \End(A)\\
    P & \mapsto (0:P\to P,P)
\end{align*}
is exact and the cokernel of the induced map on the 0th algebraic $K$-theory \cite{Q72}
\[K_0(\Mod_A^c)\to K_0(\End(A))\]
is denoted  $\widetilde{K}_0(\End(A))$ and  called the {\bf reduced $K$-theory} of $\End(A)$.
Thus,  for $A$ commutative, \cref{structure on End} implies $\widetilde{K}_{0}(\End(A))$ is a ring with group endomorphisms $F^i, V^i$.

\begin{remark*}
We use $F^i$ and $V^i$ for both the functor on the endomorphism category and the induced homomorphism on $K$-theory.  Similarly, we use the same notation to denote the functor in \eqref{eq:K_0_compose_1} and the homomorphism induced on $K$-theory.  
\end{remark*}

In addition to the characteristic polynomial, 
$\mathrm{ch}: \widetilde{K}_{0}(\End(A))\to W(A)$,  
there is an iterated trace homomorphism 
\begin{equation} \label{classical_ghost}
    \mathrm{It Tr}: \widetilde{K}_{0}(\End(A)) \to \Pi^{\infty}A
\end{equation}
induced by a map from $\End(A)$ by sending 
$(f,P)$ to $(\tr(f), \tr(f^2), \tr(f^3), \cdots)$ where $\tr$ is the Hattori-Stallings trace.

\begin{thm} [See \cite{DRESS19891}]
  The map $ch: \widetilde{K}_{0}(\End(A))\to W(A)$ is a ring map that respects the Frobenius and Verschiebung. 
  Further, it factors through the homomorphism in 
  \eqref{classical_ghost}.
\[\begin{tikzcd}
	{\tilde{K}_0 (\mathrm{End}(A)) } && {W(A)} \\
	& {\prod^\infty A}
	\arrow["{\mathrm{ch}}", from=1-1, to=1-3]
	\arrow["{\mathrm{ItTr}}"', from=1-1, to=2-2]
	\arrow[from=2-2, to=1-3]
\end{tikzcd}\]
\end{thm}

\section{$K$-theory of endomorphisms}\label{sec:K-end}

In this section we establish essential facts about a generalization of the category of endomorphisms in  \cref{motivating example}.  This is the context where we will 
define the Frobenius and Verschiebung in \cref{sec:new_maps_frobenius,sec:new_maps_verscheibung}.
We postpone the proofs of some results in this section to 
\cref{appendix}
since the techniques  used are very different from those we use in the rest of the paper.

For a category $\mathcal{C}$,
 functors $G,H\colon \mathcal{C}_{}\to \mathcal{C}$,
and a vector $\vec{i}=(i_0,\ldots, i_{n-1})$ of natural numbers, 
the objects of the category
$\Mor (\mathcal{C};G,H;\vec{i})$\label{def:mor} are
\begin{itemize}
    \item  tuples $(c_0,c_1,\dots,c_n)$ of objects in $\mathcal{C}$ and  morphisms $\left(f_j\colon G^{i_j}(c_{j+1})\to H^{i_j}(c_j)\right)_{j=0}^{n-1}.$ 
    \end{itemize}
    The morphisms 
\begin{itemize}
\item  $h\colon\left(f_j\colon G^{i_j}(c_{j+1})\to H^{i_j}(c_j)\right)_{j=0}^{n-1}\to \left(g_j\colon G^{i_j}(d_{j+1})\to H^{i_j}(d_j)\right)_{j=0}^{n-1}$ are tuples $(h_j:c_j\to d_j)_{j=0}^{n}$
so  the square
\[\begin{tikzcd}
	{G^{i_j}(c_{j+1})} & {H^{i_j}(c_{j})} \\
	{G^{i_j}(d_{j+1})} & { H^{i_j}(d_{j})}
	\arrow["{f_{j}}", from=1-1, to=1-2]
	\arrow["{G^{i_j}(h_{j+1})}"', from=1-1, to=2-1]
	\arrow["{H^{i_j}(h_{j})}", from=1-2, to=2-2]
	\arrow["{g_{j}}"', from=2-1, to=2-2]
\end{tikzcd}\]
    commutes for $j=0,1,\dots, n-1$.
    \end{itemize}

    We usually omit the tuple of objects of $\mathcal{C}$ and write an object of $\Mor(\mathcal{C};G,H;\vec{i})$ as \[\left(f_j\colon G^{i_j}(c_{j+1})\to H^{i_j}(c_j)\right)_{j=0}^{n-1}\quad\text{or} \quad(f_0,f_1,\dots, f_{n-1}).\]

Let \[\End(\mathcal{C};G,H;\vec{i})\] 
be the subcategory 
of 
$\Mor(\mathcal{C};G,H;\vec{i})$ where $c_0=c_n$ for objects and $h_0=h_n$ for morphisms.
We visualize the objects of $\End(\mathcal{C};G,H;\vec{i})$ as in \cref{fig:obj_of_end}.  (Compare to \cref{fig:classical_frob}.)
\begin{figure}
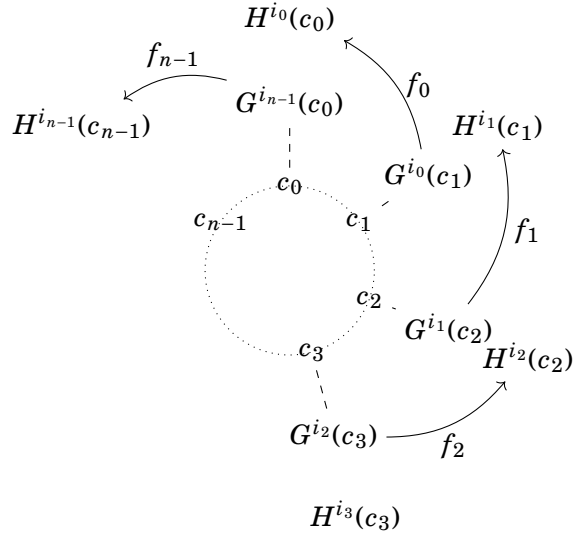

\tikz{
\tikzmath{\x1=55;}
\tikzmath{\x2=.75;}
\draw [dotted] (0,0) circle (1.5*\x2);

\node (cn1) at ({90+\x1}:1.5*\x2) {$c_{n-1}$};
\node (c0) at (90:1.5*\x2) {$c_0$};
\node (c1) at (90-\x1:1.5*\x2) {$c_1$};
\node (c2) at (90-2*\x1:1.5*\x2) {$c_2$};
\node (c3) at (90-3*\x1:1.5*\x2) {$c_3$};
\node (fc0) at (90:3*\x2) {$G^{i_{n-1}}(c_0)$};
\node (fc1) at (90-1*\x1:3*\x2) {$G^{i_0}(c_1)$};
\node (fc2) at (90-2*\x1:3*\x2) {$G^{i_1}(c_2)$};
\node (fc3) at (90-3*\x1:3*\x2) {$G^{i_2}(c_3)$};

\node (hcn1) at ({90+\x1}:4.5*\x2) {$H_{}^{i_{n-1}}(c_{n-1})$};
\node (hc0) at (90:4.5*\x2) {$H^{i_0}(c_0)$};
\node (hc1) at (90-1*\x1:4.5*\x2) {$H^{i_1}(c_1)$};
\node (hc2) at (90-2*\x1:4.5*\x2) {$H^{i_2}(c_2)$};
\node (hc3) at (90-3*\x1:4.5*\x2) {$H^{i_3}(c_3)$};

\draw [<-] (hc0)
edge [bend left=25] node[midway, right] {$f_0$} (fc1) ;
\draw  (c0)
edge [dashed] (fc0);
\draw [<-] (hc1)
edge [bend left=25] node[midway, right] {$f_1$} (fc2) ;
\draw  (c1)
edge [dashed] (fc1);
\draw [<-] (hc2)
edge [bend left=25] node[midway,  below] {$f_2$} (fc3) ;
\draw  (c2)
edge [dashed] (fc2);
\draw  (c3)
edge [dashed] (fc3);
\draw [<-] (hcn1)
edge [bend left=25] node[midway, above] {$f_{n-1}$} (fc0) ;

}
\caption{An object of $\End(\mathcal{C};G,H;\vec{i})$}\label{fig:obj_of_end}
\end{figure}

\begin{example}\label{def;end_examples} 
    The two examples above motivate these definitions.
    \begin{enumerate}
        \item\label{def:end_single_ring}  For a commutative ring $A$, we have $\End(A)$ as in \cref{classical_End(A)}. Then 
        \[\End(A) = \End(\Mod_A^c;\id,\id;(1)).\]  
        
        \item \label{def;end_modules}
        Let $R$ and $S$ be (possibly noncommutative) rings.  
        An  $R$-$R$-bimodule
$M$ and $S$-$S$-bimodule $N$ define functors  $G\coloneqq M\otimes_R -$, and $H\coloneqq -\otimes_S N$.   
Then $\End(R,S;M;N)$ from  \cref{motivating example} 
is
 \[\End(\Mod_{R,S}^c;M\otimes_R -, -\otimes_S N;(1)).\]
 \end{enumerate}
\end{example}

\begin{remark*}
    There are many possible generalizations of $\End(\mathcal{C};G,H;\vec{i})$.  For example, we could allow for different vectors of exponents on the $G$'s and $H$'s or we could allow for vectors $G_0,\ldots , G_{n-1}$ and similarly for $H$. 

     In the interest of not further over-complicating notation, we have opted for a generality that better matches the focus in the later sections. 
\end{remark*}

With some additional assumptions, the endomorphism category acquires more structure.
\begin{restatable}{lemma}{whoknowsintro}
\label{whoknows_intro}
    Let $\mathcal{C}$ be an abelian category with $G\colon\mathcal{C}\to\mathcal{C}$  right exact and $H\colon \mathcal{C}\to \mathcal{C}$ left exact. Then 
    $\End(\mathcal{C};G,H;\vec{i})$ 
    is an abelian category.
\end{restatable}
See \cref{appendix} for the proof of \cref{whoknows_intro}.

\begin{remark*}This lemma implies no additional conditions on \cref{def;end_examples}\eqref{def:end_single_ring}, but in \cref{def;end_examples}\eqref{def;end_modules},
 the functor \(H = - \otimes_S N\) is left exact only if $N$ is flat. 
This condition is restrictive, but it ensures  \(\End(\mathcal{C}; G, H; \vec{i})\) forms an abelian category. 
\end{remark*}

For 
an abelian  category $\mathcal{C}$, the product $\mathcal{C}^{\times n}$
is also abelian.  
There is an exact functor 
\begin{equation}\label{C^n_trivial_inclusion}
\mathcal{C}^{\times n}\to \End(\mathcal{C};G,H;\vec{i}) 
\end{equation}
given on objects by 
\[(c_0,c_1,\ldots , c_{n-1})\mapsto 
 \Bigl(G^{i_0}(c_1)\xrightarrow{0}H^{i_0}(c_0), G^{i_1}(c_2) \xrightarrow{0}H^{i_1}(c_1),\ldots ,G^{i_{n-1}}(c_0) \xrightarrow{0}H^{i_{n-1}}(c_{n-1})\Bigr),
\] where 0 denotes the zero map in $\mathcal{C}$. 
This functor induces a well-defined homomorphism 
\begin{equation}\label{eq:include_zero}
K_0(
\mathcal{C}^{\times n}
)\to K_0(\mathcal{C};G,H;\vec{i}),
\end{equation} where \(K_0(\mathcal{C}; G, H; \vec{i})\) denotes the zeroth \(K\)-theory of the abelian category \(\End(\mathcal{C}; G, H; \vec{i})\).

\begin{definition}\label{def:reduced_k_theory}
    For $\mathcal{C}$, $G$, and $H$ as in \cref{whoknows_intro}, the {\bf zeroth reduced $K$-theory} of $\End(\mathcal{C};G,H;\vec{i}),$ denoted
$\widetilde{K}_0(\mathcal{C};G,H;\vec{i}),$
is the cokernel of the homomorphism in 
\eqref{eq:include_zero}.
\end{definition}

Following the convention in \cref{def;end_examples}\eqref{def;end_modules}, 
    $K_0(R,S; M,N)\coloneqq K_0(\Mod_{R,S}^c;G,H; (1))$
and 
\begin{equation}\label{lem:modules_functors}   \widetilde{K}_0(R,S; M,N)\coloneqq \widetilde{K}_0(\Mod_{R,S}^c;G,H;(1)).
\end{equation}

Let $\vec{i}=(i_0,i_1)$ be a pair of natural numbers. An isomorphism $GH \cong HG$ defines  a functor 
\begin{equation}\label{eq:def:gammaprime}
    \Gamma '\colon \Mor (\mathcal{C};G,H;(i_0,i_1))\to \Mor (\mathcal{C};G,H;(i_0+i_1)),
\end{equation}
 which sends $(G^{i_0}(c_1)\xrightarrow{f_0}H^{i_0}(c_0),G^{i_1}(c_2)\xrightarrow{f_1}H^{i_1}(c_1))$ to 
    \[G^{i_0+i_1}(c_2)\xrightarrow{G^{i_0}(f_1)}G^{i_0}H^{i_1}(c_1)\cong H^{i_1}G^{i_0}(c_1)\xrightarrow{H^{i_1}(f_0)}H^{i_0+i_1}(c_0),\]
and similarly for morphisms. This extends recursively to arbitrary $\vec{i}=(i_0,\dots, i_{n-1})$. Define
\begin{equation}\label{gamma'}\Gamma '\colon \Mor(\mathcal{C};G,H;(i_0,\dots, i_{n-1}))\to\Mor(\mathcal{C};G,H;(i_0+i_1+\cdots +i_{n-1}))\end{equation}
by 
$\Gamma '(f_0,\dots,f_{n-1})=\Gamma '(f_0,\Gamma '(f_1,(\Gamma'(\cdots f_{n-3},\Gamma '(f_{n-2},f_{n-1}))\cdots)))$
and similarly for morphisms.

We will be most interested in the case where the vector
\[\vec{i}=(\underbrace{1,1,\ldots ,1}_{n-1\text{ copies}}, k)\] has length $n$.  In this case, we denote
$\End(\mathcal{C};G,H;\vec{i})$ by 
$\End(\mathcal{C};G,H;n,k)$ and 
 use the  corresponding notations for the associated $K$-theory and reduced $K$-theory. 
    The objects of the category 
    $\End(\mathcal{C};G,H;n,1)$  are tuples of maps $\left(f_j\colon G(c_{j+1})\to H(c_j)\right)_{j=0}^{n-1}$.  An object of  
    $\End(\mathcal{C}; G,H;1,n )$ is a  map $G^n(c)\to H^n(c).$
The functor in \eqref{gamma'} restricts to  an exact functor 
\begin{equation}\label{eq:K_0_compose_1}
\begin{aligned} 
\Gamma\colon \End(\mathcal{C};G,H;n,1)\to
\End(\mathcal{C}; G,H;1,n )
=
\End(\mathcal{C}; G^{ n},H^{ n};1,1 )
\end{aligned}
\end{equation}
that composes maps.
This induces a homomorphism on reduced $K$-theory
\begin{equation}\label{eq:uncompose_2}
\widetilde{K}_0(\mathcal{C};G,H;n,1) \to \widetilde{K}_0(\mathcal{C}; G,H;1,n ).
\end{equation}
The following result is a consequence  of \cref{lem:uncompose_specific}.
\begin{cor}\label{lem:multiple_uncompose}
      Let $\mathcal{C}$,
$G$, $H$ be as in \cref{whoknows_intro}.
If either $G=1_\mathcal{C}$ or $H=1_\mathcal{C}$, 
the homomorphism in \eqref{eq:uncompose_2} is an isomorphism.
\end{cor}

\begin{remark*}
    We have no evidence that the assumption that $G$ or $H$ is the identity is necessary for the conclusion of this result.  A stronger version of this statement seems possible and even likely.  However, we are unable to prove the result at this time. 
    
    For \cref{def;end_examples}\eqref{def;end_modules}, this means that we  consider twisted endomorphisms rather than bitwisted endomorphisms.  If \cref{lem:multiple_uncompose} holds without the assumption that $G$ or $H$ is the identity, we expect that all results later in the paper that currently require that assumption would also no longer need it.
\end{remark*}

\section{Frobenius}\label{sec:new_maps_frobenius} 

In this section, we generalize 
\cref{structure on End}\eqref{structure on End_F} to $\End(\mathcal{C};G,H;\vec{i})$. We will closely mimic the definition for $\End(A)$, leading to a definition of the Frobenius on the $K$-theory of twisted endomorphism that is easy to work with. We then show the first identity in \cref{structure on End}\eqref{structure on End_id}  continues to hold in this more general context. 

Let 
\begin{equation}\label{eq:iterate_map_2}\End(\mathcal{C};G,H;1,1) \to 
    \End(\mathcal{C};G,H;n,1)
    \end{equation}
    be the functor that sends an object $f\colon G(c)\xrightarrow{} H(c)$ to the tuple 
    \[
    \Bigl(
    G(c)\xrightarrow{f}H(c), G(c)\xrightarrow{f}H(c),\ldots , G(c)\xrightarrow{f}H(c)
    \Bigr)
    .\] 
We think of the image of $f$ as the tuple in \cref{fig:obj_of_end} where all $i_j=1$ and $f_j=f$
    and similarly for morphisms.
 Since the functor in \eqref{eq:iterate_map_2} is exact and the following diagram, where the top arrow is the diagonal,  commutes 
\[\begin{tikzcd}
	{\mathcal{C}} & {\mathcal{C}^{\times n}} \\
	{\End(\mathcal{C};G;H;1,1)    } & {\End(\mathcal{C};G;H;n,1)}
	\arrow[from=1-1, to=1-2]
	\arrow["{\eqref{C^n_trivial_inclusion}}", from=1-1, to=2-1]
	\arrow["{\eqref{C^n_trivial_inclusion}}", from=1-2, to=2-2]
	\arrow["{\eqref{eq:iterate_map_2}}", from=2-1, to=2-2]
\end{tikzcd}\]
 the functor in \eqref{eq:iterate_map_2} induces a homomorphism 
\begin{equation}\label{eq:iterate_map}
    \widetilde{K}_0(\mathcal{C};G,H;1,1)\to 
    \widetilde{K}_0(\mathcal{C};G,H;n,1).
\end{equation}

\begin{lemma}\label{lem:action}
    There is a $\mathbb{Z}/n$ action on $\widetilde{K}_0(\mathcal{C};G,H;n,1).$  
        If  $\Gamma$ in \eqref{eq:K_0_compose_1} is an isomorphism, this  defines a $\mathbb{Z}/n$ action on $\widetilde{K}_0(\mathcal{C};G,H;1, n)$. 
\end{lemma}

\begin{proof}

While we think of an element of $
\End(\mathcal{C};G,H;n,1)
$ 
as in \cref{fig:obj_of_end}, it has a specified starting point.   Cyclically rotating the starting point defines an exact endofunctor $R$  
where 
\begin{align*}R&
\Bigl(
G(c_1)\xrightarrow{f_0}H(c_0), G(c_2)\xrightarrow{f_1}H(c_1),\ldots, G(c_0)\xrightarrow{f_{n-1}}H(c_{n-1})
\Bigr)
\\
&=
\Bigl(G(c_2)\xrightarrow{f_1}H(c_1),\ldots , G(c_0)\xrightarrow{f_{n-1}}H(c_{n-1}),G(c_1)\xrightarrow{f_0}H(c_0)\Bigr)
\end{align*}
and similarly on morphisms. This descends to an endomorphism on $\widetilde{K}_0(\mathcal{C},G,H;n,1)$.
Since $R^n$ is the identity, $R$  is the generator of an action of $\mathbb{Z}/n$ on $\widetilde{K}_0(\mathcal{C},G,H;n,1).$ 
\end{proof}

In particular, the image of the homomorphism in \eqref{eq:iterate_map} is contained in 
$ \widetilde{K}_0(\mathcal{C};G,H;n,1)^{\mathbb{Z}/n}$.

\begin{figure}
\tikz{
\tikzmath{\x1=55;}
\tikzmath{\x2=.85;}

\node (gextra) at (-4.5, 0.5) {$G(c)$};
\node (hextra) at (-3, 0.5) {$H(c)$};
\draw [->] (gextra) -- (hextra) node[midway, above] {$f$};

\node (kextra) at (6.5, 0.5) {$G^{\circ n}(c)$};
\node (iextra) at (9, 0.5) {$H^{\circ n}(c)$};
\draw [->] (kextra) -- (iextra) node[midway, above] {$F^n(f)$};

\draw [|->, thick] (4.5, 0.5) -- (5.5, 0.5);

\draw [|->, thick] (-2, 0.5) -- (-1, 0.5);

\begin{scope}[shift={(1.5,0)}]

\draw [dotted] (0,0) circle (1.5*\x2);

\node (cn1) at ({90+\x1}:3*\x2) {$H(c)$};
\node (c0) at (90:3*\x2) {$H(c)$};
\node (c1) at (90-\x1:3*\x2) {$H(c)$};
\node (c2) at (90-2*\x1:3*\x2) {$H(c)$};
\node (c3) at (90-3*\x1:3*\x2) {$H(c)$};

\node (fc0) at (90:1.5*\x2) {$G(c)$};
\node (fc1) at (90-1*\x1:1.5*\x2) {$G(c)$};
\node (fc2) at (90-2*\x1:1.5*\x2) {$G(c)$};
\node (fc3) at (90-3*\x1:1.5*\x2) {$G(c)$};

\draw [<-] (c0) edge [bend left=25] node[midway, right] {$f$} (fc1);
\draw       (c0) edge [dashed] (fc0);
\draw [<-] (c1) edge [bend left=25] node[midway, right] {$f$} (fc2);
\draw       (c1) edge [dashed] (fc1);
\draw [<-] (c2) edge [bend left=25] node[midway, below] {$f$} (fc3);
\draw       (c2) edge [dashed] (fc2);
\draw       (c3) edge [dashed] (fc3);
\draw [<-] (cn1) edge [bend left=25] node[midway, above] {$f$} (fc0);

\end{scope}
}
\caption{$F^n([f])$
    }
\label{fig:repeat_obj}
\end{figure}

\begin{definition}\label{def:frobenius_alt}
The {\bf $n$-Frobenius map} $F^n$ 
is the composite 
\begin{align*}
    \widetilde{K}_0(\mathcal{C};G,H;1,1)
 &\xrightarrow{\mathrm{\eqref{eq:iterate_map}}} 
 \widetilde{K}_0(\mathcal{C};G,H;n,1)^{\mathbb{Z}/n}   
    \xrightarrow{\mathrm{\eqref{eq:uncompose_2}}}  
    \widetilde{K}_0(\mathcal{C};G,H;1,n)=\widetilde{K}_0(\mathcal{C};G^{\circ n} ,H^{\circ n};1,1 ).
\end{align*} 

\end{definition}
The intuition for the Frobenius map is illustrated in \cref{fig:repeat_obj}; compare with the classical case in \cref{fig:classical_frob}.  In cases where \eqref{eq:uncompose_2} is an isomorphism, we will regard $F^n$ as valued in $\widetilde{K}_0(\mathcal{C};G,H;1,n)^{\mathbb{Z}/n}$.

\begin{prop}\label{prop:compute_frobenius}
The image of $\left[ f\colon G(c) \to H(c)\right] $ under the $n$-Frobenius map 
is 
\[[\Gamma(\underbrace{f,f,\ldots, f}_{n-\text{copies}}) ].\]

\end{prop}

\begin{proof}
 Since the maps on $K$-theory are induced by functors of endomorphism categories,  the image of a single  endomorphism class in $K$-theory is determined by  its image at the level of endomorphism categories. 

Consider the following commutative diagram:
\[\begin{tikzcd}
	{\End(\mathcal{C};G,H;1,1)} & {\widetilde{K}_0(\mathcal{C};G,H;1,1)} \\
	{\End(\mathcal{C};G,H;n,1)} & {\widetilde{K}_0(\mathcal{C};G,H;n,1)} \\
	{\End(\mathcal{C};G,H;1,n)} & {\widetilde{K}_0(\mathcal{C};G,H;1,n)}
	\arrow[from=1-1, to=1-2]
	\arrow["{\mathrm{\eqref{eq:iterate_map_2}}}", from=1-1, to=2-1]
	\arrow["{\mathrm{\eqref{eq:iterate_map}}}", from=1-2, to=2-2]
	\arrow[from=2-1, to=2-2]
	\arrow["{\eqref{eq:K_0_compose_1}}", from=2-1, to=3-1]
	\arrow["{\mathrm{\eqref{eq:uncompose_2}}}", from=2-2, to=3-2]
	\arrow[from=3-1, to=3-2]
\end{tikzcd}\]
The left composite is given by
\[
    \bigl(f\colon G(c)\to H(c)\bigr)
    \;\mapsto\;
    \bigl(\underbrace{f,\ldots, f}_{n\text{-copies}}\bigr)
    \;\mapsto\;
    \Gamma\bigl(\underbrace{f,\ldots, f}_{n\text{-copies}}\bigr)
  \]
and so the image of a  class $[f\colon G(c)\to H(c)]$   in $\widetilde{K}_0(\mathcal{C};G,H;1,n)$ is 
$[\Gamma(\underbrace{f,\ldots, f}_{n\text{-copies}})]$.
\end{proof}

\begin{figure}[h]
    \centering
\[\begin{tikzcd}
	{G^{i_0+i_1+i_2}(c_3)} \\
	{G^{i_0+i_1}H^{i_2}(c_2)} & {G^{i_0}H^{i_2}G^{i_1}(c_2)} & {G^{i_0}H^{i_1+i_2}(c_1)} \\
	{H^{i_2}G^{i_0+i_1}(c_2)} & {H^{i_2}G^{i_0}H^{i_1}(c_1)} & {H^{i_1+i_2}G^{i_0}(c_1)} \\
	&& {H^{i_0+i_1+i_2}(c_0)}
	\arrow["{G^{i_0+i_1}(f_2)}"', from=1-1, to=2-1]
	\arrow["\cong", from=2-1, to=2-2]
	\arrow["\cong"', from=2-1, to=3-1]
	\arrow["{G^{i_0}H^{i_2(f_1)}}", from=2-2, to=2-3]
	\arrow["\cong"', from=2-2, to=3-1]
	\arrow["\cong", from=2-3, to=3-2]
	\arrow["\cong", from=2-3, to=3-3]
	\arrow["{H^{i_2}G^{i_0}f_1}"', from=3-1, to=3-2]
	\arrow["\cong"', from=3-2, to=3-3]
	\arrow["{H^{i_1+i_2}(f_0)}"', from=3-3, to=4-3]
\end{tikzcd}\]
    \caption{$\Gamma'(\Gamma '(f_0,f_1),f_2)=\Gamma '(f_0,\Gamma '(f_1,f_2))$}
    \label{figure_gamma}
\end{figure}

Recall the category $\Mor(\mathcal{C};G,H;\vec{i})$ from page \pageref{def:mor} and $\Gamma'$ from \eqref{eq:def:gammaprime}.

\begin{lemma}\label{gamma_coherence}
  For an object $(f_0,\dots, f_{n-1})$ of $\Mor (\mathcal{C};G,H;(i_0,\dots, i_{n-1})),$ and  integers $t_0\dots, t_m$ such that $t_j\geq 2$ and \(t_0 + \cdots + t_m = n\).
\begin{equation*}
    \Gamma '(f_0,f_1,\dots,f_{n-1})=
    \Gamma '
    \Bigl(\Gamma '(\underbrace{f_0,\dots, f_{t_0-1})}_{t_0\text{ objects}},\Gamma '(\underbrace{f_{t_0},f_{t_0+1},\dots ,f_{t_0+t_1-1}}_{t_1\text{ objects}}),\cdots,\Gamma '(\underbrace{f_{t_1+t_2+\cdots +t_{m-1}},\dots ,f_{n-1}}_{t_m\text{ objects}})
    \Bigr). 
\end{equation*}
\end{lemma}

\begin{proof}
    For a triple of natural numbers $\vec{i}=(i_0,i_1,i_2)$ and an object $(G^{i_j}(c_{j+1})\xrightarrow{f_j}H^{i_j}(c_j))_{j=0}^{2}$ of $\Mor(\mathcal{C};G,H;(i_0,i_1,i_2)),$  the commutative diagram in \cref{figure_gamma} implies 
\[\Gamma '(f_0,f_1,f_2)=\Gamma'(\Gamma '(f_0,f_1),f_2)=\Gamma '(f_0,\Gamma '(f_1,f_2)).\] 
The general statement follows by induction.
\end{proof}

Since $\Gamma '$ extends $\Gamma,$ a similar formula holds for $\Gamma$ when it makes sense.

\begin{cor}\label{prop:iterate_frobenius}
For natural numbers $n$ and $m$, $F^nF^m=F^{nm}.$
\end{cor}

\begin{example} Returning to 
\cref{def;end_examples}\eqref{def;end_modules}, 
the {\bf $n$-Frobenius map}  \[F^n\colon \widetilde{K}_0(R,S;M,N) \to \widetilde{K}_0(R;S;M^{\otimes n},N^{\otimes n})\]
is the  composite 
\begin{align*}
    \widetilde{K}_0(R,S;M,N)&\xlongequal{\mathrm{\eqref{lem:modules_functors}}}\widetilde{K}_0(\Mod_{R,S}^c;G,H;1,1)
    \\
    &\xrightarrow{\mathrm{\cref{def:frobenius_alt}}}
    \widetilde{K}_0(\Mod_{R,S}^c;G,H;1,n)
    \\
    &\xlongequal{\mathrm{\eqref{lem:modules_functors}}} \widetilde{K}_0(R,S;M^{\otimes n};N^{\otimes n})
\end{align*}

\cref{prop:compute_frobenius} implies the image of $\left[ f\colon M\otimes_R P\to P\otimes_S N\right] $ under the $n$-Frobenius map is
\begin{center}
\resizebox{\textwidth}{!}
{
$\left[ 
M^{\otimes_R n} \otimes_R P\rightarrow \!\cdots \!\rightarrow 
M^{\otimes_R 2}\otimes_R P\otimes_S N^{\otimes_S n-2} 
\xrightarrow{\id_M\otimes f\otimes \id_N^{n-2}} 
M\otimes_R P\otimes_S N^{\otimes_S n-1 }
\xrightarrow{f\otimes \id_N^{n-1}} 
P\otimes_S N^{\otimes_S n}\right].
$
}
\end{center}
When we restrict to $\End(A)$ for a commutative ring $A$, this implies the Frobeinus map agrees with the one defined in \cite[p. 319]{ALM3}.

\end{example}

\section{Verschiebung}\label{sec:new_maps_verscheibung} 

In this section, we define a Verschiebung map on $\End(\mathcal{C};G,H;\vec{i})$ and show that it satisfies the second and third properties in \cref{structure on End}\eqref{structure on End_id}.

 For commutative rings, a key property of the Frobenius and Verschiebung maps   is that 
\[F^nV^n(P,f)=(P^{\oplus n}, f^{\oplus n})\]
and we think of this composition as ``multiplication by $n$''.  We don't expect such a simple relation in general, but 
the Verschiebung map will  be a homomorphism 
\[V^n\colon\widetilde{K}_0(\mathcal{C};G,H;1,n)\to \widetilde{K}_0(\mathcal{C};G,H;1,1), \]
such that $F^nV^n$ has an interpretation as ``multiplication by $n$'' (\cref{lem:fnvn}). 

We require the functors $G, H\colon \mathcal{C}\to \mathcal{C}$  to preserve direct sums and   define a functor 
\begin{equation}\label{eq:def_sigma_mor}
\sigma \colon \Mor(\mathcal{C};G,H;n,1) \to \Mor(\mathcal{C};G,H;1,1)
\end{equation}
that sends an object \((f_j\colon G(c_{j+1}) \to H(c_j))_{j=0}^{n-1}\) to the composite
\begin{align*}
G\left( \bigoplus_{j=0}^{n-1} c_{j+1} \right)
&\xrightarrow{\oplus_{j=0}^{n-1} G(\pi_{j+1})} \bigoplus_{j=0}^{n-1} G(c_{j+1}) 
\xrightarrow{\oplus_{j=0}^{n-1} f_j} \bigoplus_{j=0}^{n-1} H(c_j) 
\xrightarrow{\oplus_{j=0}^{n-1} H (l_j)} H\left(\bigoplus_{j=0}^{n-1}c_j \right),
\end{align*}
where \(\pi_j\) and \(l_j\) denote the canonical projections and inclusions. 

For objects $c_0,\ldots, c_{n-1}$,
let 
\begin{equation}\label{eq:define_twist}
t\colon \bigoplus_{j=0}^{n-1} c_j\to 
 \bigoplus_{j=0}^{n-1} c_{j+1}
\end{equation}
be the natural isomorphism in \(\mathcal{C}\) that rotates the summands of the direct sum by moving the first component to the end.  (The same isomorphism appears in \cref{structure on End}\eqref{structure on End_V}.)

\begin{remark}
    Note that we reduce the subscripts of $c_i$ mod $n$.  We will continue this convention in the rest of the paper.    
\end{remark}

Define a functor   
\begin{equation}\label{eq:summing_map}
\begin{tikzcd}[row sep =-.05in]
	{\End(\mathcal{C};G,H;n,1)} & {\End(\mathcal{C};G,H;1,1)} \\
	{\left(G(c_1)\xrightarrow{f_0}H(c_0),G(c_2)\xrightarrow{f_1}H(c_1) , \ldots , G(c_0)\xrightarrow{f_{n-1}}H(c_{n-1}) \right)} & {\sigma(f_0, \dots, f_{n-1}) \circ G(t)}
	\arrow[from=1-1, to=1-2]
	\arrow[maps to, from=2-1, to=2-2]
\end{tikzcd}
\end{equation} 
Morphisms  are similarly mapped by applying  
$G(t)$ and then taking the direct sum of the componentwise morphisms.
\begin{figure}
\tikz{
\tikzmath{\x1=1.6;}
\tikzmath{\x2=.65;}

\def\ParenL{\resizebox{.06in}{.7in}{$\Bigg($}}
\def\ParenR{\resizebox{.06in}{.7in}{$\Bigg)$}}

\node (Gleftlabel) at (-4.5*\x1, -3.5*\x2) {$G$};
\node (Gleftlpar) at (-4.3*\x1, -3.5*\x2) {\ParenL};
\node (Gleftc0) at (-4*\x1, -2*\x2) {$c_0$};
\node (Gleftc1) at (-4*\x1, -3*\x2) {$c_1$};
\node (Gleftvdots) at (-4*\x1, -4*\x2) {$\vdots$};
\node (Gleftcn1) at (-4*\x1, -5*\x2) {$c_{n-1}$};
\node (Gleftrpar) at (-3.7*\x1, -3.5*\x2) {\ParenR};

\node (Grightlabel) at (-2.5*\x1, -3.5*\x2) {$G$};
\node (Grighlpar) at (-2.3*\x1, -3.5*\x2) {\ParenL};
\node (Grightc1) at (-2*\x1, -2*\x2) {$c_1$};
\node (Grightvdots) at (-2*\x1, -3*\x2) {$\vdots$};
\node (Grightcn1) at (-2*\x1, -4*\x2) {$c_{n-1}$};
\node (Grightc0) at (-2*\x1, -5*\x2) {$c_0$};
\node (Grightrpar) at (-1.7*\x1, -3.5*\x2) {\ParenR};

\draw [->, thick] (-3.6*\x1, -3.5*\x2) -- (-2.7*\x1, -3.5*\x2) node[midway, above] {$G(t)$};

\node (fc0) at (0*\x1, -2*\x2) {$G(c_1)$};
\node (f1)  at (0*\x1, -3*\x2) {$\vdots$};
\node (fc1) at (0*\x1, -4*\x2) {$G(c_{n-1})$};
\node (fc2) at (0*\x1, -5*\x2) {$G(c_0)$};

\node (gc0) at (1.5*\x1, -2*\x2) {$H(c_0)$};
\node (g1)  at (1.5*\x1, -3*\x2) {$\vdots$};
\node (gc1) at (1.5*\x1, -4*\x2) {$H(c_{n-2})$};
\node (gc2) at (1.5*\x1, -5*\x2) {$H(c_{n-1})$};

\draw [->] (fc0) -- (gc0) node[midway, above] {$f_0$};
\draw [->] (fc1) -- (gc1) node[midway, above] {$f_1$};
\draw [->] (fc2) -- (gc2) node[midway, above] {$f_2$};

\node (hlabel) at (3*\x1, -3.5*\x2) {$H$};
\node (hlpar) at (3.2*\x1, -3.5*\x2) {\ParenL};
\node (hc0) at (3.5*\x1, -2*\x2) {$c_0$};
\node (h1) at (3.5*\x1, -3*\x2) {$\vdots$};
\node (hc1) at (3.5*\x1,-4*\x2) {$c_{n-2}$};
\node (hc2) at (3.5*\x1,-5*\x2) {$c_{n-1}$};
\node (hrpar) at (3.8*\x1, -3.5*\x2) {\ParenR};

\draw [->, thick] (-1.5*\x1, -3.5*\x2) -- (-.6*\x1, -3.5*\x2) node[midway, above] {$\cong$};

\draw [->,thick] (2*\x1, -3.5*\x2) -- (2.8*\x1, -3.5*\x2) node[midway, above] {$\cong$};

}
\caption{Intuition for the image of $V^n(f)$ where $f = \Gamma(f_0, f_1, \dots, f_{n-1})$.}
\end{figure}

Since the functor in  \eqref{eq:summing_map}  is exact and the following diagram, where the top row sums, commutes 
\[\begin{tikzcd}
	{\mathcal{C}^{\times n}} && {\mathcal{C}} \\
	{\End(\mathcal{C};G;H;n,1)} && {\End(\mathcal{C};G;H;1,1)    }
	\arrow[from=1-1, to=1-3]
	\arrow["{\eqref{C^n_trivial_inclusion}}", from=1-1, to=2-1]
	\arrow["{\eqref{C^n_trivial_inclusion}}", from=1-3, to=2-3]
	\arrow["{\eqref{eq:summing_map}}"', from=2-1, to=2-3]
\end{tikzcd}\] \eqref{eq:summing_map} induces a  homomorphism 
    \begin{equation}\label{eq:summing_map_2}
    \widetilde{K}_0(\mathcal{C};G,H;n,1)
    \to 
    \widetilde{K}_0(\mathcal{C};G,H;1,1).\end{equation}

\begin{definition}\label{defn:verschiebung_alt}
If $G$ or $H$ is the identity, the {\bf $n$-Verschiebung map} $V^n$ is the composite 
    \begin{align*}
\widetilde{K}_0(\mathcal{C};G^{\circ n} ,H^{\circ n};1,1 )=
\widetilde{K}_0(\mathcal{C};G,H;1,n)
    &\xleftarrow[\cong]
    {\mathrm{\eqref{eq:uncompose_2}}}  
    \widetilde{K}_0(\mathcal{C};G,H;n,1)
\xrightarrow{\mathrm{\eqref{eq:summing_map_2}}} 
\widetilde{K}_0(\mathcal{C};G,H;1,1).
\end{align*}
\end{definition}

\begin{example}\label{defn:verschiebung} Returning to \cref{def;end_examples}\eqref{def;end_modules},
    the {\bf $n$-Verschiebung map}
    \[V^n\colon  \widetilde{K}_0(R,S;M^{\otimes_R n},N^{\otimes_S n})\to \widetilde{K}_0(R,S; M,N)\]
    is the composite 
    \begin{align*}
    \widetilde{K}_0(R,S;M^{\otimes_R n},N^{\otimes_S n})&\stackrel{\mathrm{\eqref{lem:modules_functors}}}{=}
\widetilde{K}_0(\Mod_{R,S}^c;G,H;1,n)
\\
&\xrightarrow{\mathrm{\cref{defn:verschiebung_alt}}}
\widetilde{K}_0(\Mod_{R,S}^c;G,H;1,1)
\\
&\stackrel{\mathrm{\eqref{lem:modules_functors}}}{=} \widetilde{K}_0(R,S; M,N)
\end{align*}

    If \(A\) is a commutative ring and $M=N=A$,
this map coincides with the classical Verschiebung map from   \cref{fig:classical_versh}, since \eqref{eq:uncompose_2} has an inverse on the level of endomorphisms, given by
    \[[f\colon P\to P]\mapsto [f\colon P\to P,\id\colon P\to P,\ldots ,\id\colon P\to P].  \]
\end{example}

We don't have an analog of \cref{prop:compute_frobenius} for the Verschiebung because we neither possess nor expect an explicit description of the inverse of the isomorphism in \eqref{eq:uncompose_2}.  We do have an analog of \cref{prop:iterate_frobenius}, but we have to prove it directly.

\begin{prop}\label{prop:vnvm}
For natural numbers $n$ and $m$, $V^nV^m=V^{nm}.$
\end{prop}

\begin{example}\label{ex:vnvm}
    To develop intuition,  consider the example where $n=3$, $m=2$ and $G=H=\id$.  (In this case we could factor an endomorphism $f$ using $f$ and 5 identity maps, but to develop intuition, we consider a more general factorization.)
    
    Let $f\colon c\to c$ and suppose $f=g_0\circ g_1\circ \cdots \circ  g_5$ for $g_0\colon c_1\to c$, $g_1\colon c_2\to c_1$, 
     $g_2\colon c_3\to c_2$, $g_3\colon c_4\to c_3$, $g_4\colon c_5\to c_4$, $g_5\colon c\to c_5$.   
    Then $f=(g_0\circ g_1\circ g_2) \circ (g_3\circ g_4\circ  g_5)$ and $V^2(f)$ is as in \cref{fig:v2}.  The maps in \cref{fig:v2preimage} are a preimage for $V^2(f)$ under $\Gamma$.  Then $V^3V^2(f)$ is 
    \cref{fig:v3v2}.  By coherence $V^3V^2(f)=V^6(f)$.
\end{example}
  \begin{figure}
    \centering
    \begin{minipage}[b]{0.47\textwidth}
      \begin{subfigure}[b]{\linewidth}
      \centering
	   \tikz{
\tikzmath{\x1=.65;}
\tikzmath{\x2=.65;}

\node (fc0) at (-1*\x1, 2*\x2) {$c$};
\node (fc1) at (-1*\x1,1*\x2) {$c_3$};

\node (gc0) at (1*\x1, 2*\x2) {$c_3$};
\node (gc1) at (1*\x1,1*\x2) {$c$};

\node (hc0) at (3*\x1, 2*\x2) {$c_2$};
\node (hc1) at (3*\x1,1*\x2) {$c_5$};

\node (kc0) at (5*\x1, 2*\x2) {$c_1$};
\node (kc1) at (5*\x1,1*\x2) {$c_4$};

\node (lc0) at (7*\x1, 2*\x2) {$c$};
\node (lc1) at (7*\x1,1*\x2) {$c_3$};

\draw [->] (fc0)
edge 
(gc1) ;

\draw [-, line width=2mm, white] (fc1)
edge  node[midway, above] {$f$} (gc0) ;

\draw [->] (fc1)
edge  
(gc0) ;

\draw [->] (gc0)
edge  node[midway, above] {$g_2$} (hc0) ;

\draw [->] (gc1)
edge  node[midway, above] {$g_5$} (hc1) ;

\draw [->] (hc0)
edge  node[midway, above] {$g_1$} (kc0) ;

\draw [->] (hc1)
edge  node[midway, above] {$g_4$} (kc1) ;

\draw [->] (kc0)
edge  node[midway, above] {$g_0$} (lc0) ;

\draw [->] (kc1)
edge  node[midway, above] {$g_3$} (lc1) ;
}
\caption{$V^2(f)$}\label{fig:v2}
      \end{subfigure}\\[\baselineskip]
      \begin{subfigure}[b]{\linewidth}
	   \tikz{
\tikzmath{\x1=.65;}
\tikzmath{\x2=.65;}
\node (fc0) at (-1*\x1, 2*\x2) {$c$};
\node (fc1) at (-1*\x1,1*\x2) {$c_3$};

\node (gc0) at (1*\x1, 2*\x2) {$c_3$};
\node (gc1) at (1*\x1,1*\x2) {$c$};

\node (hc0) at (3*\x1, 2*\x2) {$c_2$};
\node (hc1) at (3*\x1,1*\x2) {$c_5$};

\node (dhc0) at (3*\x1+\x1, 2*\x2) {$c_2$};
\node (dhc1) at (3*\x1+\x1,1*\x2) {$c_5$};

\node (kc0) at (5*\x1+\x1, 2*\x2) {$c_1$};
\node (kc1) at (5*\x1+\x1,1*\x2) {$c_4$};

\node (dkc0) at (5*\x1+2*\x1, 2*\x2) {$c_1$};
\node (dkc1) at (5*\x1+2*\x1,1*\x2) {$c_4$};

\node (lc0) at (7*\x1+2*\x1, 2*\x2) {$c$};
\node (lc1) at (7*\x1+2*\x1,1*\x2) {$c_3$};

\draw [->] (fc0)
edge  node[midway, above] {} (gc1) ;

\draw [-, line width=2mm, white] (fc1)
edge  node[midway, above] {$f$} (gc0) ;

\draw [->] (fc1)
edge  
(gc0) ;

\draw [->] (gc0)
edge  node[midway, above] {$g_2$} (hc0) ;

\draw [->] (gc1)
edge  node[midway, above] {$g_5$} (hc1) ;

\draw [->] (dhc0)
edge  node[midway, above] {$g_1$} (kc0) ;

\draw [->] (dhc1)
edge  node[midway, above] {$g_4$} (kc1) ;

\draw [->] (dkc0)
edge  node[midway, above] {$g_0$} (lc0) ;

\draw [->] (dkc1)
edge  node[midway, above] {$g_3$} (lc1) ;

}
\caption{$\Gamma^{-1}(V^2(f))$}\label{fig:v2preimage}
      \end{subfigure}
    \end{minipage}
        \hspace{.5in}
        \begin{subfigure}[b]{0.3\textwidth}
\tikz{
\tikzmath{\x1=.65;}
\tikzmath{\x2=.65;}

\node (ec0) at (-3*\x1, 2*\x2) {$c$};
\node (ec1) at (-3*\x1, 1*\x2) {$c_3$};
\node (ec2) at (-3*\x1, -2*\x2) {$c_2$};
\node (ec3) at (-3*\x1, -3*\x2) {$c_5$};
\node (ec4) at (-3*\x1, 0*\x2) {$c_1$};
\node (ec5) at (-3*\x1, -1*\x2) {$c_4$};

\node (fc0) at (-1*\x1, -2*\x2) {$c$};
\node (fc1) at (-1*\x1,-3*\x2) {$c_3$};

\node (gc0) at (1*\x1, -2*\x2) {$c_3$};
\node (gc1) at (1*\x1,-3*\x2) {$c$};

\node (hc0) at (3*\x1, -2*\x2) {$c_2$};
\node (hc1) at (3*\x1,-3*\x2) {$c_5$};

\node (dhc0) at (-1*\x1, 0*\x2) {$c_2$};
\node (dhc1) at (-1*\x1,-1*\x2) {$c_5$};

\node (kc0) at (3*\x1, 0*\x2) {$c_1$};
\node (kc1) at (3*\x1,-1*\x2) {$c_4$};

\node (dkc0) at (-1*\x1, 2*\x2) {$c_1$};
\node (dkc1) at (-1*\x1,1*\x2) {$c_4$};

\node (lc0) at (3*\x1, 2*\x2) {$c$};
\node (lc1) at (3*\x1,1*\x2) {$c_3$};

\draw [->] (ec0)
edge[out=-45,in=135]   node[midway, above] {} (fc0) ;
\draw [->] (ec1)
edge [out=-45,in=135]  node[midway, above] {} (fc1) ;

\draw [-, line width=2mm, white] (ec2)
edge  
(dhc0) ;

\draw [->] (ec2)
edge  
(dhc0) ;

\draw [-, line width=2mm, white] (ec3)
edge  
(dhc1) ;

\draw [->] (ec3)
edge  
(dhc1) ;

\draw [-, line width=2mm, white] (ec4)
edge  
(dkc0) ;

\draw [->] (ec4)
edge  
(dkc0) ;

\draw [-, line width=2mm, white] (ec5)
edge  
(dkc1) ;

\draw [->] (ec5)
edge  
(dkc1) ;

\draw [->] (fc0)
edge  node[midway, above] {} (gc1) ;

\draw [-, line width=2mm, white] (fc1)
edge  node[midway, above] {$f$} (gc0) ;

\draw [->] (fc1)
edge  
(gc0) ;

\draw [->] (gc0)
edge  node[midway, above] {$g_2$} (hc0) ;

\draw [->] (gc1)
edge  node[midway, above] {$g_5$} (hc1) ;

\draw [->] (dhc0)
edge  node[midway, above] {$g_1$} (kc0) ;

\draw [->] (dhc1)
edge  node[midway, above] {$g_4$} (kc1) ;

\draw [->] (dkc0)
edge  node[midway, above] {$g_0$} (lc0) ;

\draw [->] (dkc1)
edge  node[midway, above] {$g_3$} (lc1) ;

}
    \caption{$V^3V^2(f)$}\label{fig:v3v2}
    \end{subfigure}
\caption{\cref{ex:vnvm}}\label{fig:overall:v3v2}
  \end{figure}

\begin{proof}[Proof of \cref{prop:vnvm}]

We first work at the level of endomorphisms and define maps $\alpha$ and $\beta$ so that the regions of the  diagram in \cref{VnVm_square} commute up to isomorphism.  Then passing to reduced $K$-theory will give the statement of the theorem.

Define a map \[\alpha:\End(\mathcal{C};G,H;nm,1)\to \End(\mathcal{C};G^{\circ n},H^{\circ n};m,1)\]
as
    \[\alpha(f_0, f_1, \ldots, f_{nm-1})
\coloneqq (\Gamma '(f_0,f_1, \ldots, f_{n-1}), \Gamma '(f_n,f_{n+1}, \ldots , f_{2n-1}), \ldots , \Gamma '(f_{(m-1)n}, \ldots ,f_{mn-1}))\]
where, for $j=0,\dots, m-1,$
\[\Gamma'(f_{jn},f_{jn+1},\dots,f_{(j+1)n-1}):G^n(c_{(j+1)n})\to H^n (c_{jn}).\]
By \cref{gamma_coherence}, the top triangle in \cref{VnVm_square} commutes.

\begin{figure}
\begin{tikzcd}
	{\End(\mathcal{C};G,H;1,nm)} \\
	& {\End({\mathcal{C}};{G^{\circ n}},{H^{\circ n}};m,1)} \\
	{\End(\mathcal{C};G,H;nm,1)} & {\End(\mathcal{C};G,H;1,n)} \\
	& {\End(\mathcal{C};G,H;n,1)} \\
	{\End(\mathcal{C};G,H;1,1)}
	\arrow["\Gamma"', from=2-2, to=1-1]
	\arrow["{\eqref{eq:summing_map}}", from=2-2, to=3-2]
	\arrow["\Gamma", from=3-1, to=1-1]
	\arrow["\alpha", from=3-1, to=2-2]
	\arrow["\beta"', from=3-1, to=4-2]
	\arrow["{\eqref{eq:summing_map}}"', from=3-1, to=5-1]
	\arrow["\Gamma"', from=4-2, to=3-2]
	\arrow["{\eqref{eq:summing_map}}", from=4-2, to=5-1]
\end{tikzcd}
\caption{Diagram for the proof of \cref{prop:vnvm}}\label{VnVm_square}
\end{figure}

The functor $\beta$ is defined as in \eqref{eq:def_sigma_mor}, except we sum groups of $m$ maps whose indices differ by a multiple of $n$. 
Explicitly $\beta$ is defined by
\begin{align*}
    \beta(f_0, f_1, &\ldots , f_{nm-1})
    \\&\coloneqq 
(\sigma(f_0, f_n, \ldots ,f_{n(m-1)} )
, \sigma(f_1, f_{n+1}, \ldots ,f_{n(m-1)+1}), \ldots, 
\sigma(f_{n-1}, f_{2n-1}, \ldots , f_{nm-1})\circ Gt)
\end{align*}
where for $i=0,\dots ,n-1,$
$$\sigma(f_i,f_{n+i},\dots, f_{n(m-1)+i})\colon G(c_{i+1}\oplus c_{n+i+1}\oplus \cdots\oplus c_{n(m-1)+i+1}) \to H(c_i\oplus c_{n+i}\oplus \cdots\oplus c_{n(m-1)+i}),$$
and so,
\[\sigma(f_{n-1},f_{2n-1},\dots ,f_{mn-1} )\circ Gt\colon G(c_0\oplus c_n\oplus\cdots \oplus c_{(m-1)n})\to H(c_{n-1}\oplus c_{2n-1}\oplus \cdots \oplus c_{mn-1}).\]

The middle square then compares
\begin{equation}\label{eq:middle_square_1}\sigma(\Gamma '(f_0, \ldots, f_{n-1}), \Gamma '(f_n, \ldots , f_{2n-1}), \ldots , \Gamma '(f_{n(m-1)}, \ldots f_{nm-1})\circ Gt)
\end{equation}
and 
\begin{equation}\label{eq:middle_square_2}\Gamma(\sigma(f_0, f_n, \ldots ,f_{n(m-1)} )
, \sigma(f_1, f_{n+1}, \ldots ,f_{n(m-1)+1}), \ldots, 
\sigma(f_{n-1}, f_{2n-1}, \ldots , f_{nm-1})).
\end{equation}
Both arrows have domain $G^n(c_0\oplus \cdots\oplus c_{(m-1)n})$ and codomain $H^n(c_0\oplus c_n\oplus \cdots\oplus c_{(m-1)n}).$ The equality of these arrows follows from the fact that $G$ and $H$ are symmetric monoidal functors and the naturality of the isomorphism $GH\to HG.$ Finally, the bottom triangle commutes up to isomorphism by coherence for the  symmetric monoidal category $\mathcal{C},$ and coherence for the symmetric monoidal functors $G,H$.
\end{proof}

To state our result about $F^nV^n$, we will need to recall the transfer from group cohomology.

\begin{definition}\label{def:transfer} \cite[Transfer Maps 6.7.16]{weibel:homological}
    For an abelian group $A$ with an action of a group $G$ and  a finite index normal subgroup $H$ of $G$, the {\bf transfer}   $\tau\colon A^H\to A^G$ is the homomorphism defined by $\tau(a)=\oplus_{[gH]} ga$.  
\end{definition}

\begin{figure}[h]
\tikz{
\tikzmath{\x1=.65;}
\tikzmath{\x2=.65;}

\node (ec0) at (-3*\x1, -2*\x2) {$c_0$};
\node (ec1) at (-3*\x1, -3*\x2) {$c_1$};
\node (ec2) at (-3*\x1, -4*\x2) {$c_2$};

\node (fc0) at (-1*\x1, -2*\x2) {$c_1$};
\node (fc1) at (-1*\x1,-3*\x2) {$c_2$};
\node (fc2) at (-1*\x1,-4*\x2) {$c_0$};

\node (gc0) at (1*\x1, -2*\x2) {$c_0$};
\node (gc1) at (1*\x1,-3*\x2) {$c_1$};
\node (gc2) at (1*\x1,-4*\x2) {$c_2$};

\node (hc0) at (3*\x1, -2*\x2) {$c_1$};
\node (hc1) at (3*\x1, -3*\x2) {$c_2$};
\node (hc2) at (3*\x1, -4*\x2) {$c_0$};

\node (ic0) at (5*\x1, -2*\x2) {$c_0$};
\node (ic1) at (5*\x1,-3*\x2) {$c_1$};
\node (ic2) at (5*\x1,-4*\x2) {$c_2$};

\node (jc0) at (7*\x1, -2*\x2) {$c_1$};
\node (jc1) at (7*\x1,-3*\x2) {$c_2$};
\node (jc2) at (7*\x1,-4*\x2) {$c_0$};

\node (kc0) at (9*\x1, -2*\x2) {$c_0$};
\node (kc1) at (9*\x1,-3*\x2) {$c_1$};
\node (kc2) at (9*\x1,-4*\x2) {$c_2$};

\draw [->] (ec0)
edge[out=-25,in=155]   node[midway, above] {} (fc2) ;

\draw [-, line width=2mm, white] (ec1)
edge  
(fc0) ;

\draw [->] (ec1)
edge  
(fc0) ;

\draw [-, line width=2mm, white] (ec2)
edge  
(fc1) ;

\draw [->] (ec2)
edge  
(fc1) ;

\draw [->] (fc0)
edge  node[midway, above] {$f_0$} (gc0) ;
\draw [->] (fc1)
edge  node[midway, above] {$f_1$} (gc1) ;
\draw [->] (fc2)
edge  node[midway, above] {$f_2$} (gc2) ;

\draw [->] (gc0)
edge[out=-25,in=155]   node[midway, above] {} (hc2) ;

\draw [-, line width=2mm, white] (gc1)
edge  
(hc0) ;

\draw [->] (gc1)
edge  
(hc0) ;

\draw [-, line width=2mm, white] (gc2)
edge  
(hc1) ;

\draw [->] (gc2)
edge  
(hc1) ;

\draw [->] (hc0)
edge  node[midway, above] {$f_0$} (ic0) ;
\draw [->] (hc1)
edge  node[midway, above] {$f_1$} (ic1) ;
\draw [->] (hc2)
edge  node[midway, above] {$f_2$} (ic2) ;

\draw [->] (ic0)
edge [out=-25,in=155]  node[midway, above] {} (jc2) ;

\draw [-, line width=2mm, white] (ic1)
edge  
(jc0) ;

\draw [->] (ic1)
edge  
(jc0) ;

\draw [-, line width=2mm, white] (ic2)
edge  
(jc1) ;

\draw [->] (ic2)
edge  
(jc1) ;

\draw [->] (jc0)
edge  node[midway, above] {$f_0$} (kc0) ;
\draw [->] (jc1)
edge  node[midway, above] {$f_1$} (kc1) ;
\draw [->] (jc2)
edge  node[midway, above] {$f_2$} (kc2) ;

}
    \caption{\cref{lem:fnvn} for $n=3$ and $G=H=\id$}
    \label{eg:fnvn}
\end{figure}

\begin{lemma}\label{lem:fnvn}
    For a natural number $n$,
    $F^nV^n:\widetilde{K}_0(\mathcal{C};G,H;1,n)\to \widetilde{K}_0(\mathcal{C};G,H;1,n)^{\mathbb{Z}/n}$
    is the transfer for  the group action in \cref{lem:action} and the sum in \eqref{eq:biprod}. 
\end{lemma}
\cref{eg:fnvn} is an illustration of this result in the case $n=3$ and $G=H=\id$.

\begin{proof}
We first consider the case of $F^mV^n$.  

Let $(f_0,\dots ,f_{n-1})\in \End(\mathcal{C};G,H;n,1)$.  The image under \eqref{eq:summing_map}, \eqref{eq:iterate_map_2}, and  \eqref{eq:K_0_compose_1} is 
\[\Gamma(\sigma(f_0,\dots ,f_{n-1})\circ G(t), \ldots , \sigma(f_0,\dots ,f_{n-1})\circ G(t)).\]
Since $G$ and $H$ are strong symmetric monoidal functors (with respect to the monoidal product given by $\oplus$), and $t$ and $GH\cong HG$ are natural transformations, this composite is 
\begin{align}\label{proof:lem:fnvn}
    {G^m\left(\bigoplus_{i=0}^{n-1} c_i\right)}
    \xrightarrow{G^m(t^m)}
     {G^m\left(\bigoplus_{i=0}^{n-1} c_{m+i}\right)}
     \xrightarrow{\sigma \left(\Gamma'\Bigl(
 \underbrace{f_i,f_{i+1},\dots, f_{i+m-1}}_{m\text{ terms}} 
 \Bigr)_{i=0}^{n-1}\right)}
	 {H^m\left(\bigoplus_{i=0}^{n-1} c_i\right).}
\end{align}
The  commutative diagram in \cref{fig:fnvn} shows $F^mV^n([\Gamma(f_0,\dots,f_{n-1})])$ is the class of \eqref{proof:lem:fnvn}.

 If $m$ is multiple of $n$, then $t^m$ is the identity, and $\Gamma'=\Gamma$.  If $m=n,$ then at the level of 
 $\widetilde{K}_0$ the image of $[\Gamma(f_0,f_1,\ldots f_{n-1})]$ under $F^nV^n$ is 
\[
\sum_{i=0}^{n-1}\left[\Gamma(R^i(f_0,\dots, f_{n-1})\right] =\sum_{i=0}^{n-1}R^i\left[\Gamma(f_0,\dots,f_{n-1})\right]
    = \tau(\left[\Gamma(f_0,\dots,f_{n-1})\right]),
\]
where $R$ represents a generator of the $\mathbb{Z}/n$ action on $\widetilde{K}_0(\mathcal{C};G,H;1,n)$ and $\widetilde{K}_0(\mathcal{C};G,H;1,n)$.
\end{proof}

\begin{figure}\begin{tikzcd}
	{\End(\mathcal{C};G,H;1,n)} & {\widetilde{K}_0(\mathcal{C};G,H;1,n)} \\
	{\End(\mathcal{C};G,H;n,1)} & {\widetilde{K}_0(\mathcal{C};G,H;n,1)} \\
	{\End(\mathcal{C};G,H;1,1)} & {\widetilde{K}_0(\mathcal{C};G,H;1,1)} \\
	{\End(\mathcal{C};G,H;m,1)} & {\widetilde{K}_0(\mathcal{C};G,H;m,1)} \\
	{\End(\mathcal{C};G,H;1,m)} & {\widetilde{K}_0(\mathcal{C};G,H;1,m)}
	\arrow[from=1-1, to=1-2]
	\arrow["{\eqref{eq:K_0_compose_1}}", from=2-1, to=1-1]
	\arrow[from=2-1, to=2-2]
	\arrow["{{\mathrm{\eqref{eq:summing_map}}} }"', from=2-1, to=3-1]
	\arrow["{\eqref{eq:uncompose_2}}", from=2-2, to=1-2]
	\arrow["{{\mathrm{\eqref{eq:summing_map_2}}} }"', from=2-2, to=3-2]
	\arrow[from=3-1, to=3-2]
	\arrow["{\mathrm{\eqref{eq:iterate_map_2}}}"', from=3-1, to=4-1]
	\arrow["{{\mathrm{\eqref{eq:iterate_map}}} }"', from=3-2, to=4-2]
	\arrow[from=4-1, to=4-2]
	\arrow["{\eqref{eq:K_0_compose_1}}"', from=4-1, to=5-1]
	\arrow["{\eqref{eq:uncompose_2}}"', from=4-2, to=5-2]
	\arrow[from=5-1, to=5-2]
\end{tikzcd}
\caption{}\label{fig:fnvn}
\end{figure}
\section{Compatibility with trace}
\label{sec:define_ghost}

While the definitions of the Frobenius and Verschiebung maps in \cref{sec:new_maps_frobenius,sec:new_maps_verscheibung} directly generalize those in \cite{ALM1}, in this section we give further justification for our definitions by showing they behave as expected with respect to the \ghost map. The first step in this process is to generalize the \ghost map of (\ref{classical_ghost}).  We focus on the cases of \cref{def;end_examples} and use the bicategorical trace of \cite{ponto}. (This restriction is primarily for exposition -- the results here can be applied when the bicategorical trace is additive \cite{ps:linearity} and $G$ and $H$ are given by bicategorical composition.)  

Let $\mathcal{B}$ be a bicategory with composition of 1-cells denoted $\odot$ and unit 1-cells $U_A$.  See  \cite{leinster} for the full definition of a bicategory.

\begin{definition}\cite[16.4.1]{may2004parametrized}  A 1-cell $P\in \mathcal{B}(A,B)$ is {\bf dualizable} if there is a 1-cell 
$P^*\in \mathcal{B}(B,A)$
and 2-cells 
$\eta\colon U_A\to P\odot P^*$ and $ \epsilon\colon P^*\odot P\to U_B$ such that the triangle diagrams commute.  
\end{definition}

A symmetric monoidal category $\mathcal{C}$ can be regarded as a bicategory with a single 0-cell. In that case, if $P$ is a dualizable object, $\eta$ and $\epsilon$ can be composed using the symmetry isomorphism to define the trace of any endomorphism in $\mathcal{C}$, see \cite{Dold1978-ol}. In a bicategory, this is no longer the case, and we need to impose additional structure.

\begin{definition}\cite[Definition 4.4.1]{ponto}
    Let $\mathcal{B}$ be a bicategory. 
    A \textbf{shadow functor} for $\mathcal{B}$ consists of functors \[\sh{\text{-}}: \mathcal{B}(R, R) \to \mathcal{T}\] for each object $R$ of $\mathcal{B}$ and some fixed category $\mathcal{T}$, together with natural isomorphisms \[\theta: \sh{M \odot N} \overset{\cong}{\to} \sh{N \odot M} \] for $M: R \to S$ and $N: S \to R$ that are compatible with unit and associativity isomorphisms.
\end{definition}

\begin{example}\label{Hochschild}
The objects of the bicategory $\mathcal{M}od/_{\mathcal{R}ing}$ are  (not necessarily commutative) rings with unit. For rings $R,S$, $\mathcal{M}od/_{\mathcal{R}ing}(R,S)$ is the category of $R$-$S$-bimodules. Composition, denoted by $\odot$, is the tensor product over the common ring, with the unit $U_R$ being the $R$-$R$-bimodule $R.$ The dualizable 1-cells in $\mathcal{M}od/_{\mathcal{R}ing}(R,S)$ are precisely the objects of $\text{Mod}_{R,S}^c$. We can define a shadow for this bicategory, $(\sh{\text{-}}, \mathcal{A}b)$, which sends an $R-R$-bimodule $M$ to the abelian group $\sh{M}=M/(rm\sim mr)$.  This  is the 0th Hochschild homology of $M$.
\end{example}

\begin{definition}\cite[Definition 4.5.1]{ponto} 
\label{defn:bicat_trace}
    Let $\mathcal{B}$ be a bicategory with a shadow functor and $M$ a dualizable 1-cell of $\mathcal{B}$. The \textbf{bicategorical trace} of a 2-cell $f: Q\odot M \Rightarrow M \odot P$ is the composite: 
    \[\xymatrix{\sh{Q} \ar[r]^-{\sh{\id\odot \eta}} & \sh{Q \odot M \odot M^\ast} \ar[r]^-{\sh{f\odot \id}} & \sh{M \odot P \odot M^\ast} \ar[r]^\theta & \sh{M^\ast \odot M\odot P} \ar[r]^-{\sh{\epsilon\odot \id}} & \sh{P}}.\]
\end{definition}

The composition of the Frobenius map and the bicategorical trace of \cref{Hochschild} yields a function: 
\begin{equation}\label{eq:trace_as_functor}
\begin{aligned}
    \ob(\End(R,S;M,N))&\to \Map(\sh{M^{\odot n}}, \sh{N^{\odot n}})
    \\
    M\odot P\xto{f} P\odot M&\mapsto \tr(F^n(f))
\end{aligned}
\end{equation}
for each natural number $n$.
The bicategorical trace is additive under additional assumptions on $\mathcal{B}$ and $\mathcal{T}$ \cite{ps:linearity}, but in \cref{ap:add} we give a more direct proof for the function in \eqref{eq:trace_as_functor}.

\begin{restatable}{cor}{restatedefntraceequiv}
\label{lem:defn_trace_equiv}
The functions in \eqref{eq:trace_as_functor} induce homomorphisms,
\[\tr_n\colon \widetilde{K}_0(R,S;M,N)\to \Map(\sh{M^{\odot n}}, \sh{N^{\odot n}})\]
where the mapping space is taken in the category of abelian groups.
\end{restatable}

\begin{definition}\label{def:ghost}
    The {\bf \ghost map} \[\tr_\bullet \colon \widetilde{K}_0(R,S;M,N)\to \prod_{n\geq 1}\Map(\sh{M^{\odot n}}, \sh{N^{\odot n}})\]
    is the product of the homomorphisms defined in \cref{lem:defn_trace_equiv}.
\end{definition}

\subsection{Equivariance and the trace} \label{subsec:equiv_trace}
  For composable 1-cells $M_0,M_1,\ldots, M_{n-1}$, the cyclicity  isomorphism $\theta$ and associators define 
an isomorphism $\varsigma_{M_0,M_1,\ldots, M_{n-1}}$
\[  \sh{M_0\odot (M_1\odot (M_2\odot  \cdots\odot (M_{n-2}\odot  M_{n-1}))\cdots )}\to \sh{M_1\odot (M_2\odot (M_3\odot  \cdots\odot (M_{n-1}\odot  M_{0}))\cdots )}.\]
By \cite[Theorem 6.4]{mp:coherence}, the composition of \( n \) such rotations is the identity:
\[\varsigma_{ M_{n-1},M_0,M_1, \ldots,M_{n-2} }\circ \cdots \circ \varsigma_{M_1,\ldots, M_{n-1},M_0}\circ\varsigma_{M_0,M_1,\ldots, M_{n-1}}=\id.\] 
 In particular, if $M=M_0=M_1=\cdots =M_{n-1}$, 
$\varsigma_M: =\varsigma_{M,M,\ldots, M}$
is the action of the generator for a $\mathbb{Z}/n$ action on $\sh{
M^{\odot n}}$. 

These actions induce an action on the mapping space
\(
\Map\left(\sh{M^{\odot n}}, \sh{N^{\odot n}}\right),
\)
defined for \( \varsigma_-^{i} \in \mathbb{Z}/n \) and \( f \colon \sh{M^{\odot n}} \to \sh{N^{\odot n}} \) by:
\begin{equation} \label{Action_on_trace}
\sh{M^{\odot n}} \xrightarrow{{\varsigma_M^{-i}}} \sh{M^{\odot n}} \xrightarrow{f} \sh{N^{\odot n}} \xrightarrow{\varsigma_N^i} \sh{N^{\odot n}}.
\end{equation}
The fixed points of this action are exactly the \( \mathbb{Z}/n \)-equivariant maps:
\[
\Map\left(\sh{M^{\odot n}}, \sh{N^{\odot n}}\right)^{\mathbb{Z}/n} 
=
\Map_{\mathbb{Z}/n}\left(\sh{M^{\odot n}}, \sh{N^{\odot n}}\right).
\]

\begin{lemma}\label{gamma_trace} The composite
\[\widetilde{K}_0(\Mod_{R,S}^c; M\odot-,-\odot N;n,1)\xrightarrow{\mathrm{\eqref{eq:uncompose_2}}} \widetilde{K}_0(R,S;M^{\odot n},N^{\odot n})\xrightarrow{\tr_1}\Map(\sh{M^{\odot n}},\sh{N^{\odot n}})\]
is equivariant.
\end{lemma}
\begin{proof}
Let $(f_0,\dots, f_{n-1})$ be an object of $\End(\Mod_{R,S}^c; M\odot-,-\odot N;n,1).$ By the ``tightening" property of the trace \cite[Proposition 7.1]{Ponto_2012} we have that
\[\begin{tikzcd}
	{\sh{M^{\odot n}}} &[40pt] {\sh{N^{\odot n}}} \\
	{\sh{M^{\odot n}}} & {\sh{N^{\odot n}}}
	\arrow["{\tr ([\Gamma(f_0,f_1,\dots,f_{n-1})])}", from=1-1, to=1-2]
	\arrow["{\varsigma_M}"', from=1-1, to=2-1]
	\arrow["{\varsigma_N}", from=1-2, to=2-2]
	\arrow["{\tr ([\Gamma (f_1,\dots, f_{n-1},f_0)])}"', from=2-1, to=2-2]
\end{tikzcd}\]
commutes. Thus, if we denote the preferred generator of $\mathbb{Z}/n$  by $R$, 
\[R \tr([\Gamma(f_0,f_1,\dots, f_{n-1})])=\tr([\Gamma (f_1,\dots, f_{n-1},f_0)])=\tr([\Gamma (R(f_0,f_1,\dots, f_{n-1}))).    \qedhere\]
\end{proof}

    By \cref{gamma_trace}, for $[f]\in \widetilde{K}_0(R,S;M,N),$ $\tr ([ \Gamma(f,\dots, f)])\colon\sh{M^{\odot n}}\to \sh{N^{\odot n}}$ is a fixed point of the action of $\mathbb{Z}/n$ on $\Map (\sh{M^{\odot n}},\sh{N^{\odot n}}),$ i.e., it is equivariant and we have the following statement. 
\begin{cor}\label{equivariant_trace}
    The map $\tr_n$ factors through the inclusion 
    \[\Map_{\mathbb{Z}/n}(\sh{M^{\odot n}}, \sh{N^{\odot n}})
    \to \Map(\sh{M^{\odot n}}, \sh{N^{\odot n}}).\]
\end{cor}

\cref{equivariant_trace} and \eqref{eq:trace_as_functor},  define homomorphisms 
\[\tr_n\colon \widetilde{K}_0(R,S;M,N)\to \Map_{\mathbb{Z}/n}(\sh{M^{\odot n}}, \sh{N^{\odot n}}),\]
and 
\begin{equation} \label{equivariant_ghost}
\tr_\bullet \colon \widetilde{K}_0(R,S;M,N)\to \prod_{n\geq 1}\Map_{\mathbb{Z}/n}(\sh{M^{\odot n}}, \sh{N^{\odot n}}).
\end{equation}

\subsection{Compatibility of Frobenius and trace}
Let \[\phi_{n,i}\colon  {\Map_{\mathbb{Z}/{ni}}(\sh{N^{\odot ni}}, \sh{M^{\odot ni}})}\to  {\Map_{\mathbb{Z}/i}(\sh{(N^{\odot n})^{\odot i}}, \sh{(M^{\odot n})^{\odot i}}})\]
be the map that
restricts the action along the inclusion 
$\mathbb{Z}/i\to \mathbb{Z}/{ni}$.
Let 
\[\Phi_n\colon {\prod_{i\geq 1}\Map_{\mathbb{Z}/{i}}(\sh{N^{\odot i}}, \sh{M^{\odot i}})} 
\to {\prod_{i\geq 1}\Map_{\mathbb{Z}/i}(\sh{(N^{\odot n})^{\odot i}}, \sh{(M^{\odot n})^{\odot i}})}\]
be the product of the maps 
\[{\prod_{i\geq 1}\Map_{\mathbb{Z}/{i}}(\sh{N^{\odot i}}, \sh{M^{\odot i}})}
\xrightarrow{\mathrm{proj}_{ni}} \Map_{\mathbb{Z}/{ni}}(\sh{N^{\odot ni}}, \sh{M^{\odot ni}})
\xrightarrow{\phi_{n,i}} {\Map_{\mathbb{Z}/i}(\sh{(N^{\odot n})^{\odot i}}, \sh{(M^{\odot n})^{\odot i}})}.\]
Intuitively, 
$\Phi_n(f_1, f_2,f_3, \ldots ) = (f_n, f_{2n}, f_{3n}, \ldots) $.

\begin{thm}\label{thm:frobeius_ghost}
For $n\geq 1$, the following diagram commutes.
\[\begin{tikzcd}
	{\widetilde{K}_0(R,S;M,N)} & {\displaystyle \prod_{i\geq 1}\Map_{\mathbb{Z}/i}(\sh{M^{\odot i}}, \sh{N^{\odot i}})} \\
	{\widetilde{K}_0(R,S;M^{\odot n},N^{\odot n})} & {\displaystyle \prod_{i\geq 1}\Map_{\mathbb{Z}/i}(\sh{(M^{\odot n})^{\odot i}}, \sh{(N^{\odot n})^{\odot i}})}
	\arrow["{\tr_{\bullet}}", from=1-1, to=1-2]
	\arrow["{F^n}", from=1-1, to=2-1]
	\arrow["{\Phi_n}", from=1-2, to=2-2]
	\arrow["{\tr_\bullet}", from=2-1, to=2-2]
\end{tikzcd}\]
\end{thm}

\begin{proof}
    It is enough to prove that the diagram commutes after projecting to each of the factors in the product at the bottom right corner.
\[\begin{tikzcd}[column sep=-.25in, row sep =0.15in]
	{\widetilde{K}_0(R,S;M,N)} && {\widetilde{K}_0(R,S;M^{\odot ni},N^{\odot ni})} \\
	& {\prod_{i\geq 1}\Map_{\mathbb{Z}/i}(\sh{M^{\odot i}}, \sh{N^{\odot i}})} \\
	{\widetilde{K}_0(R,S;M^{\odot n},N^{\odot n})} && {\Map_{\mathbb{Z}/ni}(\sh{M^{\odot ni}}, \sh{N^{\odot ni}})} \\
	& {\prod_{i\geq 1}\Map_{\mathbb{Z}/i}(\sh{(M^{\odot n})^{\odot i}}, \sh{(N^{\odot n})^{\odot i}})} \\
	{\widetilde{K}_0(R,S;(M^{\odot n})^{\odot i},(N^{\odot n})^{\odot i})} && {\Map_{\mathbb{Z}/i}(\sh{(M^{\odot n})^{\odot i}}, \sh{(M^{\odot n})^{\odot i}})}
	\arrow["{F^{ni}}", from=1-1, to=1-3]
	\arrow["{\tr_{\bullet}}", from=1-1, to=2-2]
	\arrow["{F^n}", from=1-1, to=3-1]
	\arrow["\tr_1", from=1-3, to=3-3]
	\arrow["{\mathrm{proj}_{ni}}", from=2-2, to=3-3]
	\arrow["{\Phi_n}", from=2-2, to=4-2]
	\arrow["{\tr_\bullet}", from=3-1, to=4-2]
	\arrow["{F^i}"', from=3-1, to=5-1]
	\arrow["{\phi_{n,i}}", from=3-3, to=5-3]
	\arrow["{\mathrm{proj}_{i}}", from=4-2, to=5-3]
	\arrow["\tr_1"', from=5-1, to=5-3]
\end{tikzcd}\]

The triangles commute by definition of $\tr_\bullet,$ and the right square commutes by definition of $\Phi_n$.  The outside square commutes by definition of $\phi_{n,i}$ and \cref{prop:iterate_frobenius}.
\end{proof}

\subsection{Compatibility of Verschiebung and trace}
For $\mathbb{Z}/k \subseteq \mathbb{Z}/nk$, the action in \eqref{Action_on_trace} satisfies 
\[\left(\Map(\sh{(M^{\odot n})^{\odot k}}, \sh{(N^{\odot n})^{\odot k}})\right)^{\mathbb{Z}/k}=
\Map_{\mathbb{Z} /k}(\sh{(M^{\odot n})^{\odot k}}, \sh{(N^{\odot n})^{\odot k}})\] 
and so the map $\tau$ in \cref{def:transfer} defines a homomorphism 
\[\tau\colon 
\Map_{\mathbb{Z} /k}(\sh{(M^{\odot n})^{\odot k}}, \sh{(N^{\odot n})^{\odot k}})\to 
\Map_{\mathbb{Z} /nk}(\sh{M^{\odot nk}}, \sh{N^{\odot nk}}).\]

The following is a generalization of \cite[Lemma 6.2]{facets}. 

\begin{lemma}\label{FnVd} 
 If $d|n,$ the following diagram commutes.
\[\begin{tikzcd}
	{\widetilde{K_0}(R,S;M^{\odot d},N^{\odot d})} & {\widetilde{K_0}(R,S;(M^{\odot n},N^{\odot n})} \\
	{\widetilde{K_0}(R,S;M,N)} & {\Map_{\mathbb{Z}/\frac{n}{d}}(\sh{(M^{\odot d})^{\odot n/d}},\sh{(N^{\odot d})^{\odot n/d}})} \\
	{\widetilde{K_0}(R,S;M^{\odot n},N^{\odot n})^{\mathbb{Z}/n}} & {\Map_{\mathbb{Z}/n}(\sh{M^{\odot n}},\sh{N^{\odot n}}).}
	\arrow["{F^{n/d}}", from=1-1, to=1-2]
	\arrow["{V^d}"', from=1-1, to=2-1]
	\arrow["\tr", from=1-2, to=2-2]
	\arrow["{F^n}"', from=2-1, to=3-1]
	\arrow["\tau", from=2-2, to=3-2]
	\arrow["\tr"', from=3-1, to=3-2]
\end{tikzcd}\]
If $d\nmid n,$ the composite
\[\widetilde{K}_0(R,S;M^{\odot d},N^{\odot d})\xrightarrow{V^d}\widetilde{K}_0(R,S;M,N)\xrightarrow{F^n}\widetilde{K}_0(R,S;M^{\odot n},N^{\odot n})\xrightarrow{\tr_1}\Map_{\mathbb{Z}/n}(\sh{M^{\odot n}},\sh{N^{\odot n}})\]
is zero.
\end{lemma}
\begin{proof}
Suppose $d|n$ and $f_j:M\odot c_{j+1}\to c_j\odot N.$ Using \eqref{proof:lem:fnvn}, $F^{n}V^d([\Gamma(f_0,f_1,\dots, f_{d-1})])$ is
\begin{equation}\label{eq:vdfn}
\sigma\Biggl(\biggl(\Gamma(\underbrace{f_j,f_{j+1},\dots, f_{j+d-1},f_j,f_{j+1},\dots,f_{j+d-1},f_j,\dots,f_{j+d-1}}_{\text{tuple }(f_j,\dots,f_{j+d-1})\text{ repeated}\frac{n}{d}\text{ times}})\biggr)_{j=0}^{d-1} \Biggr)
\end{equation}
Then 
\begin{align*}
    \tr\eqref{eq:vdfn}&\xlongequal{\mathrm{\cref{prop:add_traces_naive}}}\sum_{j=0}^{d-1}\tr [\Gamma (\underbrace{f_j,f_{j+1},\dots, f_{j+d-1},f_j,f_{j+1},\dots,f_{j+d-1},f_j,\dots,f_{j+d-1}}_{\frac{n}{d}\text{ times}})]\\
    &\xlongequal{\text{\cref{lem:action}}}\sum_{j=0}^{d-1}\tr  [R^j\Gamma(\underbrace{f_0,f_1,\dots, f_{d-1},f_0,f_1,\dots,f_{d-1},\dots,f_{d-1}}_{\frac{n}{d}\text{ times}})]
    \\
        &\xlongequal{\text{\cref{gamma_trace}}}\sum_{j=0}^{d-1}R^j\tr  [\Gamma(\underbrace{f_0,f_1,\dots, f_{d-1},f_0,f_1,\dots,f_{d-1},\dots,f_{d-1}}_{\frac{n}{d}\text{ times}})]
        \\
        &\xlongequal{\text{\cref{gamma_coherence}}}\sum_{j=0}^{d-1}R^j\tr  [ F^{n/d}\Gamma(f_0,f_1,\dots, f_{d-1})].
    \end{align*}
    
    Suppose that $d\nmid n$ and $f_j:M\odot c_{j+1}\to c_j\odot N.$ By \eqref{proof:lem:fnvn}, 
    \[F^nV^d([\Gamma(f_0,\dots,f_{d-1}) ])=\left[\sigma\left( \Gamma '(f_k,f_{k+1},\dots, f_{k+n-1})_{k=0}^{d-1}\right)\circ (\id_{M^{\otimes n}}\odot t^n)\right].\]
    For $j=0,1\dots, n-1,$ if $l_j,\pi_j$ denote the canonical inclusion and projection respectively, the following diagram commutes and the left composite is zero.
\[\begin{tikzcd}
	& {M^{\odot n}\odot c_j} \\
	{M^{\odot n}\odot (\bigoplus_{i=0}^{d-1} c_i)} && {\bigoplus_{i=0}^{d-1} (M^{\odot n}\odot c_i)} \\
	{M^{\odot n}\odot (\bigoplus_{i=0}^{d-1}c_{n+i})} && {\bigoplus_{i=0}^{d-1} (M^{\odot n}\odot c_{n+i})} \\
	\\
	{(\bigoplus_{i=0}^{d-1} c_i)\odot N^{\odot n}} && {\bigoplus_{i=0}^{d-1} (c_i\odot N^{\odot n})} \\
	& {c_j\odot N^{\odot n}}
	\arrow["{\id_{M^{\odot n}}\odot l_j}"', from=1-2, to=2-1]
	\arrow["{l_j}", from=1-2, to=2-3]
	\arrow[from=2-1, to=2-3]
	\arrow["{\id_{M^{\odot n}}\odot t^n}"', from=2-1, to=3-1]
	\arrow["{t^n}"', from=2-3, to=3-3]
	\arrow[from=3-1, to=3-3]
	\arrow["{\sigma\biggl(\Gamma '\biggl(\underbrace{f_k,f_{k+1},\ldots,f_{k+n-1}}_{n\text{ terms}}\biggr)_{k=0}^{d-1}\biggr)}"', from=3-1, to=5-1]
	\arrow["{\displaystyle\bigoplus_{k=0}^{d-1}\Gamma '\biggl(\underbrace{f_k,f_{k+1},\ldots,f_{k+n-1}}_{n\text{ terms}}\biggr)}"', from=3-3, to=5-3]
	\arrow[from=5-1, to=5-3]
	\arrow["{\pi_j\odot \id_{N^{\odot n}}}"', from=5-1, to=6-2]
	\arrow["{\pi_j}", from=5-3, to=6-2]
\end{tikzcd}\]
 By \cref{lem:sum_of_traces},  $\tr(F^nV^d([\Gamma(f_0,\dots,f_{d-1})]))=0$.
\end{proof}
Let $$B_n\colon \displaystyle{\prod_{i\geq 1}\Map_{\mathbb{Z}/i}(\sh{(M^{\odot n})^{\odot i}},\sh{(N^{\odot n})^{\odot i}})\to \prod_{i\geq 1}}\Map_{\mathbb{Z}/i}(\sh{M^{\odot i}},\sh{N^{\odot i}})$$ be the product of the following maps:
\begin{itemize}
    \item 
If $i=nk,$
\[\resizebox{\textwidth}{!}{$\prod_{i\geq 1}\Map_{\mathbb{Z}/i}(\sh{(M^{\odot n})^{\odot i}}, \sh{(N^{\odot n})^{\odot i}})
\xto{\text{proj}_k}
\Map_{\mathbb{Z} /k}(\sh{(M^{\odot n})^{\odot k}}, \sh{(N^{\odot n})^{\odot k}})
\xto{\tau}
\Map_{\mathbb{Z}/i}(\sh{M^{\odot i}}, \sh{N^{\odot i}})$}
\]
\item 
If $k\nmid i$,
$\prod_{i\geq 1}\Map(\sh{(M^{\odot n})^{\odot i}}, \sh{(N^{\odot n})^{\odot i}})
\xrightarrow{0} 
\Map_{\mathbb{Z}/i}(\sh{M^{\odot i}}, \sh{N^{\odot i}}).
$
\end{itemize}

Intuitively
\[
B_{n}\bigl((f_i)_{i \ge 1}\bigr)
  \;=\;
  \bigl(\psi_j\bigr)_{j\ge 1},
  \quad
  \text{where}
  \quad
  \psi_j
  \;=\;
  \begin{cases}
    \tau\left(  f_{\frac{j}{n}}\right), & \text{if } n \mid j\\[6pt]
    0, & \text{otherwise.}
  \end{cases}
\]

\begin{thm}\label{thm:v_ghost}The following diagram commutes.
\[\begin{tikzcd}
	{\widetilde{K}_0(R,S;M^{\odot n},N^{\odot n})} & {\displaystyle{\prod_{i\geq 1}\Map_{\mathbb{Z}/i}(\sh{(M^{\odot n})^{\odot i}}, \sh{(N^{\odot n})^{\odot i}})}} \\
	{\widetilde{K}_0(R,S;M.N)} & {\displaystyle{\prod_{i\geq 1}\Map_{\mathbb{Z}/i}(\sh{M^{\odot i}}, \sh{N^{\odot i}})}.}
	\arrow["{\tr_\bullet}", from=1-1, to=1-2]
	\arrow["{V^n}"', from=1-1, to=2-1]
	\arrow["{B_n}", from=1-2, to=2-2]
	\arrow["{\tr_{\bullet}}"', from=2-1, to=2-2]
\end{tikzcd}\]
\end{thm}

\begin{proof}
First suppose that $i=nk$ and project to $\Map_{\mathbb{Z}/i}\Map (\sh{M^{\odot i}},\sh{N^{\odot i}}).$ In the following diagram 
the triangles commute by definition of $\tr_{\bullet},$ the right square by definition of $B_n$, and the outer square by \cref{FnVd}.
\[\begin{tikzcd}[column sep = 0in]
	{\widetilde{K_0}(R,S;M^{\odot n};N^{\odot n})} && {\widetilde{K_0}(R,S;(M^{\odot n})^{\odot k}, (N^{\odot n})^{\odot k})} \\
	& {\displaystyle\prod_{j\geq 1}\Map_{\mathbb{Z}/j}(\sh{(M^{\odot n})^{\odot j}}, \sh{(N^{\odot n})^{\odot j}})} \\
	{\widetilde{K_0}(R,S;M;N)} && {\Map_{\mathbb{Z}/k}(\sh{(M^{\odot n})^{\odot k}}, \sh{(N^{\odot n})^{\odot k}})} \\
	& {\displaystyle\prod_{j\geq 1}\Map_{\mathbb{Z}/j}(\sh{M^{\odot j}}, \sh{N^{\odot j}})} \\
	{\widetilde{K_0}(R,S;M^{\odot nk}, N^{\odot nk})} && {\Map_{\mathbb{Z}/nk}(\sh{M^{\odot nk}}, \sh{N^{\odot nk}})}
	\arrow["{F^k}", from=1-1, to=1-3]
	\arrow["{\tr_\bullet}", from=1-1, to=2-2]
	\arrow["{V^n}"', from=1-1, to=3-1]
	\arrow["\tr", from=1-3, to=3-3]
	\arrow["{\mathrm{proj}_k}"', from=2-2, to=3-3]
	\arrow["{B_n}", from=2-2, to=4-2]
	\arrow["{\tr_{\bullet}}", from=3-1, to=4-2]
	\arrow["{F^{nk}}"', from=3-1, to=5-1]
	\arrow["\tau", from=3-3, to=5-3]
	\arrow["{\mathrm{proj}_{nk}}", from=4-2, to=5-3]
	\arrow["\tr"', from=5-1, to=5-3]
\end{tikzcd}\]

    If  $k\nmid i$,  \cref{FnVd} implies 
    the trace of $V^nF^i([\Gamma(f_0,\dots, f_{n-1})])$ is zero for 
    $[\Gamma(f_0,\dots , f_{n-1})]\in \widetilde{K}_0(R,S;M^{\odot n},N^{\odot n})$.
\end{proof}

\appendix
\section{$K$-theory as a shadow}
\label{appendix}

In this section we prove the results in \cref{{sec:K-end}} following the proof outlined for Proposition 14 in \cite{gepner}.\footnote{The title of this section is motivated by \cite{gepner} where this result is used to show $\tilde{K}_0(\mathcal{C};1_\mathcal{C}, HK;1,1)\cong \tilde{K}_0(\mathcal{C};1_\mathcal{C}, KH;1,1)$.
}
\whoknowsintro*

\begin{proof}
The category 
$\End(\mathcal{C};G,H;\vec{i})$
is preadditive with componentwise addition. 
Since $\mathcal{C}$ is preadditive,  composition is bilinear with respect to this sum.
Since $G$ and $H$ preserve the zero object,  the zero object of $\End(\mathcal{C};G,H;\vec{i})$ is $(G^{i_j}(0)\xrightarrow{0}H^{i_j}(0))_{j=0}^{n-1}.$

Finite biproducts are calculated pointwise, that is, if 
\[c=\left(G^{i_j}(c_{j+1})\xrightarrow{f_{j}}H^{i_j}(c_{j})\right)_{j=0}^{n-1} 
\text{ and }
d=\left(G^{i_j}(d_{j+1})\xrightarrow{g_{j}}H^{i_j}(d_{j})\right)_{j=0}^{n-1}\] 
are objects, their biproduct is
\begin{equation}\label{eq:biprod}
\resizebox{.9\textwidth}{!}{$\left(
G^{i_j}(c_{j+1}\oplus d_{j+1})\xrightarrow{G^{i_j}\pi_1\oplus G^{i_j}\pi_2}
G^{i_j}(c_{j+1})\oplus G^{i_j}(d_{j+1})
\xrightarrow{f_j\oplus g_j}
H^{i_j}(c_j)\oplus H^{i_j}(d_j) \xrightarrow{H^{i_j}k_1\oplus H^{i_j}k_2}H^{i_j}(c_j\oplus d_j)\right)_{j=0}^{n-1}$
}\end{equation}
where $\pi_1,\pi_2$ are the product projections and $l_1,l_2$ are the coproduct inclusions. 
A diagram chase shows that 
 the projections 
\[\left(c_j\oplus d_j\xrightarrow{\pi_1} c_j\right)_{j=0}^{n-1} \text{ and } \left(c_j\oplus d_j\xrightarrow{\pi_2} d_j\right)_{j=0}^{n-1}\] 
define morphisms $\pi_1:c\oplus d\to c$ and $\pi_2:c\oplus d\to d$ in $\End(\mathcal{C};G,H;\vec{i})$ satisfying the universal property of the product.
 One can similarly check that the injections $l_1:c\to c\oplus d$ given by $(l_1:c_j\to c_j\oplus d_j)_{j=0}^{n-1}$ and $l_2:d\to c\oplus d$ given by $(l_2:c_j\to c_j\oplus d_j)_{j=0}^{n-1}$ are morphisms of the category $\End (\mathcal{C};G,H;\vec{i})$ satisfying the universal property of the coproduct.

Let $(f_j)_{j=0}^{n-1}$ and  $(g_j)_{j=0}^{n-1}$ be objects in $\End(\mathcal{C}; G,H;\vec{i})$ 
and $h\colon (f_j)_{j=0}^{n-1}\to (g_j)_{j=0}^{n-1}$ be a morphism in this category.  
The kernel of $h$ in $\End (\mathcal{C};G,H;\vec{i})$ consists of $(\ker h_j )_{j=0}^{n-1}$ together with the arrows 
\[
\left(G^{i_j}(\ker h_{j+1})
\rightarrow
\ker (G^{i_j}(h_{j+1}))
\rightarrow 
\ker (H^{i_j}(h_{j}))\leftarrow H^{i_j}(\ker h_j)
 \right)_{j=0}^{n-1},\]
 where the first and last arrows are induced  by the universal property of the kernel in $\mathcal{C}$ and the last arrow is an isomorphism in $\mathcal{C}$ since $H$ is left exact.
 
A diagram chase shows
\[k \colon 
\left( G^{i_j}(\ker h_{j+1})\rightarrow H^{i_j}(\ker h_j)\right)_{j=0}^{n-1}
\to 
\left( G^{i_j}(c_{j+1})\rightarrow H^{i_j}(c_j) \right)_{j=0}^{n-1}\] 
given by the inclusions $k_j:\ker h_j\to c_j$  defines a morphism in $\End(\mathcal{C};G,H;\vec{i})$.  The pointwise universal property of the kernel in $\mathcal{C}$ and a diagram chase show that this is the kernel in $\End(\mathcal{C},G,H;\vec{i})$.  Cokernels are computed similarly. The fact that $\End(\mathcal{C};G,H;\vec{i})$ is abelian then follows from the fact that $\mathcal{C}$ is.
\end{proof}

Similar to the functor defined in \eqref{gamma'}, 
the functor in \eqref{eq:def:gammaprime}
defines a functor 
\begin{equation}\label{eq:K_0_compose_special_case}
    \Gamma_i\colon \End(\mathcal{C};G,H;n-i+1,i)\to \End (\mathcal{C};G,H;n-i,i+1)
\end{equation}
and a
homomorphism 
\begin{equation}\label{eq:uncompose}
    \begin{aligned} 
\widetilde{K}_0(\mathcal{C};G,H;n-i+1,i)\to
\widetilde{K}_0(\mathcal{C};G,H;n-i,i+1).\end{aligned}
\end{equation}

\begin{restatable}[{Following \cite[Prop. 14]{gepner}}]{thm}{thmKendcompose}
\label{lem:uncompose_specific}
For $\mathcal{C}$,
$G$, $H$ as in \cref{whoknows_intro} and 
 either $G=1_\mathcal{C}$ or $H=1_\mathcal{C}$,  the homomorphisms in \eqref{eq:uncompose} are isomorphisms.
\end{restatable}

This result is  a consequence of the following lemma.

\begin{lemma}\label{lem:composition_k_theory}
For $\mathcal{C},G,$ and $H$ as in \cref{whoknows_intro} and $G=1_\mathcal{C}$ or $H=1_\mathcal{C}$, there is an exact sequence
\[K_0(\mathcal{C})\to K_0(\mathcal{C};G,H;n-i+1,i)\xto{K_0(\Gamma_i)} K_0 (\mathcal{C};G,H;n-i,i+1).\]
\end{lemma}
\begin{proof}
We will prove the case $G=1_\mathcal{C}$. The case $H=1_\mathcal{C}$ is entirely similar. 

There is an exact functor 
\begin{equation}\label{eq:exact_inclusion}\mathcal{C}\to \End(\mathcal{C};1_{\mathcal{C}},H;n-i+1,i)\end{equation}
which sends the object $c$ to the object
\[0\xrightarrow{f_0}H(0),\dots, 0\xrightarrow{f_{n-i-2}}H(0),c\xrightarrow{f_{n-i-1}} H(0),0\xrightarrow{f_{n-i}}H^i (c)\]
with $f_0,f_1,\dots, f_{n-i}$  all being the zero map. The arrow $h:c\to d$ is sent  to the morphism given by the arrows $(1:0\to 0,\dots,1:0\to 0,h:c\to d,\dots,1:0\to 0).$

The image of $\mathcal{C}$ under the functor in \eqref{eq:exact_inclusion} is a category isomorphic to $\mathcal{C},$ so we will make no distinction between $\mathcal{C}$ and its image. The image of \eqref{eq:exact_inclusion} is the 
kernel of the exact functor $\Gamma_i$, and so it is  
 a Serre subcategory of $\End(\mathcal{C};1_{\mathcal{C}},H;n-i+1,i)$.  
 This implies there is an exact functor $\varphi$ making the following diagram commute
\[\begin{tikzcd}
	{\End(\mathcal{C};1_{\mathcal{C}},H;n-i+1,i)} & {\End(\mathcal{C};1_{\mathcal{C}},H;n-i+1,i)/\mathcal{C}} \\
	& {\End(\mathcal{C};1_{\mathcal{C}},H;n-i,i+1)}
	\arrow[from=1-1, to=1-2]
	\arrow["{\Gamma_i}"', from=1-1, to=2-2]
	\arrow["\varphi", from=1-2, to=2-2]
\end{tikzcd}\]
Since the kernel of $\Gamma_i$ is $\mathcal{C},$ the exact functor $\varphi$ is faithful \cite[\href{https://stacks.math.columbia.edu/tag/06XK}{Tag 06XK}]{stacks-project}.

The functor 
$\phi\colon \End(\mathcal{C};1_{\mathcal{C}},H;n-i,i+1)\to \End(\mathcal{C};1_{\mathcal{C}},H;n-i+1,i)$ that sends the object
\[\bigl(
c_1\xrightarrow{f_0}H(c_0),c_2\xrightarrow{f_1} H(c_1),\dots, {c_{n-i-1}}\xrightarrow{f_{n-i-2}}H(c_{n-i-2}) ,c_0\xrightarrow{f_{n-i-1}}H^{i+1}(c_{n-i-1})
\bigr)
\]
to the object
\[\resizebox{.99\textwidth}{!}{$\bigl(
c_1\xrightarrow{f_0}H(c_0),c_2\xrightarrow{f_1}H(c_1),\dots,c_{n-i-1}\xrightarrow{f_{n-i-2}}H(c_{n-i-2}),H(c_{n-i-1})\xrightarrow{1}H(c_{n-i-1}),c_0\xrightarrow{f_{n-i-1}}H^{i+1}(c_{n-i-1})
\bigr)$},\]
and the morphism $h=(h_0,h_1,\dots , h_{n-i-1})$ to $(h_0,h_1,\dots,h_{n-i-1},H(h_{n-i-1}))$ is 
a right inverse for  $\Gamma_i$, 
and thus, $\Gamma_i$ and $\varphi$ are full and surjective on objects. We conclude that $\varphi$ is an equivalence of categories. 

Quillen's localization theorem for abelian categories  \cite{Q72} applied to the equivalence of categories $\varphi$  implies there is a long exact sequence of $K$-theory groups that ends in the short exact sequence  in the statement of the lemma.
\end{proof}

\begin{proof}[Proof of \cref{lem:uncompose_specific}]

Recall the exact functor $$\underbrace{\mathcal{C}\times\cdots\times \mathcal{C}}_{n-i+1 \text{ copies}}\to \End(\mathcal{C};G,H;n-i+1,i)$$
defined in \eqref{C^n_trivial_inclusion}.
The inclusion into the last coordinate, $\mathcal{C}\to \mathcal{C}\times \dots \times \mathcal{C},$ and the projection to the first coordinates
$$\underbrace{\mathcal{C}\times \cdots\times \mathcal{C}}_{n-i\text{ copies}}\to \underbrace{\mathcal{C}\times \cdots\times \mathcal{C}}_{n-i-1\text{ copies}}$$ are also exact functors and so induce maps on $K_0.$ At the level of $K_0,$ we get the following commutative diagram.
\[\begin{tikzcd}
	{K_0(\mathcal{C})} & {{K}_0(\overbrace{\mathcal{C}\times\cdots\times \mathcal{C}}^{n-i+1 \text{ copies}})} & {{K}_0(\overbrace{\mathcal{C}\times\cdots\times \mathcal{C}}^{n-i \text{ copies}})} & 0 \\
	{K_0(\mathcal{C})} & {K_0(\mathcal{C};G,H;n-i+1,i)} & {K_0(\mathcal{C},G,H;n-i,i+1)} & 0 \\
	0 & {\widetilde{K}_0(\mathcal{C};G,H;n-i+1,i)} & {\widetilde{K}_0(\mathcal{C};G,H;n-i,i+1)} \\
	& 0 & 0
	\arrow[shift right, from=1-1, to=1-2]
	\arrow[shift right, no head, from=1-1, to=2-1]
	\arrow[shift left, no head, from=1-1, to=2-1]
	\arrow[shift right=2, from=1-2, to=1-3]
	\arrow[from=1-2, to=2-2]
	\arrow[from=1-3, to=1-4]
	\arrow[from=1-3, to=2-3]
	\arrow[from=2-1, to=2-2]
	\arrow[from=2-1, to=3-1]
	\arrow[from=2-2, to=2-3]
	\arrow[from=2-2, to=3-2]
	\arrow[from=2-3, to=2-4]
	\arrow[from=2-3, to=3-3]
	\arrow[from=3-1, to=3-2]
	\arrow[from=3-2, to=3-3]
	\arrow[from=3-2, to=4-2]
	\arrow[from=3-3, to=4-3]
\end{tikzcd}\]
The first row is exact \cite[Example II.7.1.6.]{weibel:kbook}.  The middle row is exact by \cref{lem:composition_k_theory} and the center and right columns are exact by definition of reduced $K_0$. Then the desired isomorphism follows by a diagram chase.
\end{proof}

\section{Additivity of trace}
\label{ap:add}
In this section, we give an explicit proof that the trace on the bicategory of rings, bimodules, and their homomorphisms is additive on short exact sequences. This shows that the trace map descends to $K$-theory.

For a $R$-$S$-bimodule $P$, let  $_\iota P$ denote $P$ regarded as a $\mathbb{Z}$-$S$-bimodule.  Similarly, for a $R$-$R$-bimodule $M$, let  $_\iota M_\iota$ denote $M$ regarded as an abelian group.  The quotient from the tensor over $\mathbb{Z}$ to tensor over $R$ defines a homomorphism $\rho\colon _\iota M_\iota \otimes_\mathbb{Z} \,_\iota P\to 
    \,_\iota M\otimes_RP$.  A similar quotient defines a map 
    $q\colon \sh{\,_\iota M_\iota}\to \sh{M}$.

\begin{prop}\label{prop:trace_forget_R}
For a homomorphism  $f\colon M\otimes_RP\to P\otimes_SN$, the  trace of 
\begin{equation}\label{eq:forget_R}
_\iota M_\iota \otimes_\mathbb{Z} \,_\iota P\xto{\rho}      \,_\iota M\otimes_RP\xto{_\iota f} \,_\iota P\otimes_SN
\end{equation}
is the composite 
$\sh{_\iota M_\iota}\xto{q}\sh{M}\xto{\tr(f)} \sh{N}$.
\end{prop}

\begin{proof}
There is a commutative diagram 
 \begin{equation}\label{eq:f_to_Z_mod}
\begin{tikzcd}
	{_\iota R\otimes_R M\otimes_R R_\iota \otimes_\mathbb{Z} \,_\iota R\otimes_R P} && {   \,_\iota R\otimes_R  M\otimes_RP} & {_\iota R\otimes_R  P\otimes_SN} \\
	{_\iota  M_\iota \otimes_\mathbb{Z} \,_\iota P} && { \,_\iota  M\otimes_RP} & {_\iota  P\otimes_SN}
	\arrow["{\id\otimes \id\otimes\epsilon\otimes \id}", from=1-1, to=1-3]
	\arrow["{\alpha\otimes \beta}"', from=1-1, to=2-1]
	\arrow["{\id\otimes  f}", from=1-3, to=1-4]
	\arrow["{\gamma\otimes \id}"', from=1-3, to=2-3]
	\arrow["{\beta\otimes \id}", from=1-4, to=2-4]
	\arrow[from=2-1, to=2-3]
	\arrow["{_\iota f}"', from=2-3, to=2-4]
\end{tikzcd}
\end{equation}
where the vertical maps are isomorphisms.  
Then the following diagram commutes by the cited results. 
\[\begin{tikzcd}[column sep=1.5in, row sep = .3in]
	{\sh{M\otimes_R R_\iota \otimes_\mathbb{Z} \,_\iota R }} & {\sh{M}} \\
	\\
	{\sh{_\iota R\otimes_R M\otimes_R R_\iota }} & {\sh{N}} \\
	{\sh{_\iota  M_\iota} } & {\sh{N}}
	\arrow["{\sh{\id\otimes \epsilon}}", from=1-1, to=1-2]
	\arrow["{\tr(f)}", from=1-2, to=3-2]
	\arrow["{\cong }", from=3-1, to=1-1]
	\arrow[""{name=0, anchor=center, inner sep=0}, "{\tr(\id\otimes \id\otimes\epsilon)}"{description}, from=3-1, to=1-2]
	\arrow["{\tr((\id\otimes f)(\id \otimes \id\otimes \epsilon\otimes \id))}"', from=3-1, to=3-2]
	\arrow["{\sh{\alpha}}"', from=3-1, to=4-1]
	\arrow["\text{\cite[7.1]{Ponto_2012}}"{description}, draw=none, from=3-1, to=4-2]
	\arrow["\id", from=3-2, to=4-2]
	\arrow["{\tr((_\iota f)\rho)}"', from=4-1, to=4-2]
	\arrow["\text{\cite[3.13]{morita}}"{description}, draw=none, from=1-1, to=0]
	\arrow["\text{\cite[7.5]{Ponto_2012} }"{description}, draw=none, from=0, to=3-2]
\end{tikzcd}\]
Direct computation shows $q$ is the left and top composite. 
\end{proof}

  For $f\colon M\otimes (P'\oplus P'')\to  (P'\oplus P'')\otimes N$, let 
  $f'\colon M\otimes P'\to P'\otimes M$ be the composite 
  \begin{align*}
       M\otimes P'&
       \xto{\id \otimes i_{P'}}M\otimes  (P'\oplus P'')
       \xto{f}(P'\oplus P'')\otimes N
       \xto{\pi_{P'}\otimes \id}P'\otimes N
  \end{align*}
  and similarly for $f''\colon M\otimes P''\to P''\otimes N$.

\begin{prop}\label{lem:sum_of_traces}
    For $f\colon M\otimes (P'\oplus P'')\to  (P'\oplus P'')\otimes N$, $\tr(f)=\tr(f')+\tr(f'')$. 
\end{prop}

\begin{proof}
    Let $Q$, $Q'$ and $Q''$ be $\mathbb{Z}$-$S$-bimodules and $L$ be an abelian group.  Let $g\colon L\otimes_{\mathbb{Z}}Q\to Q\otimes_R N$ be  a homomorphism and similarly for $g'$ and $g''$.

    If $\eta_{Q'}$ and $\epsilon_{Q'}$ are a choice of coevaluation and evaluation for $Q'$ and similarly for $Q''$, then the following are a choice of coevaluation for $Q'\oplus Q''$ 
        \begin{align*}
        \mathbb{Z}\xto{\triangle}\mathbb{Z}\oplus \mathbb{Z}
        &\xto{\eta_{Q'}\oplus \eta_{Q''}}
        (Q'\otimes_S \Hom_S(Q',S))
        \oplus 
        (Q ''\otimes_S \Hom_S(Q'',S))
        \\&\hookrightarrow
        (Q'\oplus Q'')\otimes_S (\Hom_S(Q', S)\oplus \Hom_S(Q'', S))
    \end{align*}    
    and evaluation
 \begin{align*}
         (\Hom_S(Q', S)\oplus \Hom_S(Q'', S))\otimes_{\mathbb{Z}}  (Q'\oplus Q'')
         &\xto{\mathrm{proj}}
          (\Hom_S(Q', S){\otimes_{\mathbb{Z}}} Q')\oplus 
          (\Hom_S(Q'', S){\otimes_{\mathbb{Z}}} Q'')
          \\
          &\xto{\epsilon_{Q'}\oplus \epsilon_{Q''}}
          S\oplus S
          \xto{+}S
          \end{align*}
The triangle identities follow from a routine diagram chase.

In  \cref{fig:sum_of_traces}, 
the maps labeled with $\sim$ are natural isomorphisms.  Those labeled with $\pi$ are examples of 
\[\sh{T\oplus T}\to \sh{T}\oplus \sh{T}\]
induced by the projections.  The maps labeled $\hookrightarrow$ are inclusions and $\mathrm{proj}$ are projections. 
 Most small regions of the diagram commute by definition or naturality.  The  exception is marked with $\ast$, and its commutativity can be verified by a direct computation on generators.

In \cref{fig:sum_of_traces}, the trace $\tr(g)$ is realized by the outer composite along the top, right, and bottom edges. The sum $\tr(g') + \tr(g'')$ corresponds to the left vertical composite through the summands and their evaluations.

    If we take $L=\,_\iota M_\iota$, $Q=\,_\iota P$, $Q'=\,_\iota P'$ and $Q''=\,_\iota P''$ and $g=(_\iota f)\circ \rho$, and similarly for  $g'$ and $g''$, then \cref{prop:trace_forget_R} gives the following computation.
    \begin{align*}
    \tr(f)([m])&= \tr(f)(q[m])=\tr(g)([m])
    \\
    &=(\tr(g')+\tr(g''))([m])=\tr(g')([m])+\tr(g'')([m])
    \\
    &=\tr(f')(q[m])+\tr(f'')(q[m])
     =\tr(f')([m])+\tr(f'')([m]) \qedhere
     \end{align*}
\end{proof}

\begin{figure}
\adjustbox{scale=.65,center}
{\begin{tikzcd}[column sep=-.75in,row sep=.3in]
	& {\sh{ L\otimes_\mathbb{Z}\mathbb{Z} }} \\
	{\sh{L}} && {\sh{L\otimes_\mathbb{Z} (\mathbb{Z}\oplus \mathbb{Z})}} &&& {       \sh{ L\otimes_\mathbb{Z} ((Q'\otimes_S (Q')^*)\oplus(Q''\otimes_S (Q'')^*))}} \\
	& {\sh{L\oplus L}} \\
	{\sh{L}\oplus \sh{L}} && {\sh{(L\otimes_\mathbb{Z} \mathbb{Z})\oplus (L\otimes_\mathbb{Z} \mathbb{Z})}} \\
	& {\sh{L\otimes_\mathbb{Z} \mathbb{Z}}\oplus \sh{L\otimes_\mathbb{Z} \mathbb{Z}}} && {    \sh{ (L\otimes_\mathbb{Z}  Q'\otimes_S (Q')^*)\oplus(L\otimes_\mathbb{Z} Q''\otimes_S (Q'')^*)}} && {       \sh{ L\otimes_\mathbb{Z} (Q'\oplus Q'')\otimes_S ((Q')^*\oplus (Q'')^*)}} \\
	&& { \sh{ L\otimes_\mathbb{Z}  Q'\otimes_S (Q')^*}\oplus\sh{L\otimes_\mathbb{Z} Q''\otimes_S (Q'')^*}} && {\sh{  ((L\otimes_\mathbb{Z} Q')\oplus (L\otimes_\mathbb{Z} Q''))\otimes_S ((Q')^*\oplus (Q'')^*)}} \\
	&&& {\sh{(Q'\otimes_S N\otimes_S(Q')^*)\oplus (Q''\otimes_S N\otimes_S(Q'')^*)}} && { \sh{(Q'\oplus Q'')\otimes_S N\otimes_S ((Q')^*\oplus (Q'')^*)}} \\
	&& {\sh{ Q'\otimes_S  N\otimes_S (Q')^*}\oplus\sh{Q''\otimes_S N\otimes_S (Q'')^*}} && {\sh{((Q'\otimes_S N)\oplus (Q''\otimes_S N))\otimes_S ((Q')^*\oplus (Q'')^*)}} \\
	&&& {\sh{(Q'\otimes_S N\otimes_S(Q')^*)\oplus (Q''\otimes_S N\otimes_S(Q'')^*)}} \\
	&& {\sh{(Q')^*\otimes_\mathbb{Z} Q'\otimes_S N}\oplus \sh{(Q'')^*\otimes_\mathbb{Z} Q''\otimes_S N}} & \ast & {\sh{((Q')^*\oplus (Q'')^*)\otimes_\mathbb{Z} ((Q'\otimes_S N)\oplus ( Q''\otimes_S N))}} \\
	& {\sh{S\otimes_S N}\oplus \sh{S\otimes_S N}} && {\sh{((Q')^*\otimes_\mathbb{Z} Q'\otimes_S N)\oplus ((Q'')^*\otimes_\mathbb{Z} Q''\otimes_S N)}} && { \sh{((Q')^*\oplus (Q'')^*)\otimes_\mathbb{Z}   (Q'\oplus Q'')\otimes_S N} } \\
	{\sh{N}\oplus \sh{N}} && {\sh{(S\otimes_S N)\oplus (S\otimes_S N)}} \\
	& {\sh{N\oplus N}} \\
	{\sh{N}} && {\sh{(S\oplus S)\otimes_S N}} &&& {\sh{((Q')^*\otimes_\mathbb{Z} Q')\oplus ((Q'')^*\otimes_\mathbb{Z} Q''))\otimes_S N}} \\
	& {\sh{S\otimes_S N}}
	\arrow["\sim"{description}, from=1-2, to=2-1]
	\arrow["{\sh{\id\otimes \triangle }}", from=1-2, to=2-3]
	\arrow["{\sh{\triangle}}"{description}, from=2-1, to=3-2]
	\arrow["{\triangle }"{description}, from=2-1, to=4-1]
	\arrow["{\sh{\id \otimes (\eta_{Q'}\oplus \eta_{Q''})}}", from=2-3, to=2-6]
	\arrow[hook, from=2-6, to=5-6]
	\arrow["\pi"{description}, from=3-2, to=4-1]
	\arrow["\sim"{description}, from=4-1, to=5-2]
	\arrow["\sim"{description}, from=4-3, to=2-3]
	\arrow["\sim"{description}, from=4-3, to=3-2]
	\arrow["\pi"{description}, from=4-3, to=5-2]
	\arrow["{\sh{(\id \otimes \eta_{Q'})\oplus (\id \otimes \eta_{Q''})}}"{description}, from=4-3, to=5-4]
	\arrow["{\sh{\id \otimes \eta_{P'}}\oplus \sh{\id \otimes \eta_{P''}}}"{description}, from=5-2, to=6-3]
	\arrow["\sim"{description}, from=5-4, to=2-6]
	\arrow["\pi"{description}, from=5-4, to=6-3]
	\arrow[hook', from=5-4, to=6-5]
	\arrow["{\sh{(g'\otimes \id)\oplus (g''\otimes \id)}}"{description}, from=5-4, to=7-4]
	\arrow["{\sh{g\otimes \id} }", from=5-6, to=7-6]
	\arrow["{\sh{g'\otimes \id}\oplus \sh{g''\otimes \id}}"{description}, from=6-3, to=8-3]
	\arrow["\sim"{description}, from=6-5, to=5-6]
	\arrow["{\sh{(g'\oplus g'')\otimes \id}}"{description}, from=6-5, to=8-5]
	\arrow["\pi"{description}, from=7-4, to=8-3]
	\arrow[hook', from=7-4, to=8-5]
	\arrow["\id"{description}, from=7-4, to=9-4]
	\arrow["\theta", from=7-6, to=11-6]
	\arrow["{\theta\oplus \theta}", from=8-3, to=10-3]
	\arrow["\sim"{description}, from=8-5, to=7-6]
	\arrow["{\sh{\mathrm{proj}}}"{description}, from=8-5, to=9-4]
	\arrow["\theta", from=8-5, to=10-5]
	\arrow["\pi"{description}, from=9-4, to=8-3]
	\arrow["{\sh{\epsilon_{Q'}\otimes \id}\oplus \sh{\epsilon_{Q''}\otimes \id} }"{description}, from=10-3, to=11-2]
	\arrow["{\sh{\mathrm{proj}}}", from=10-5, to=11-4]
	\arrow["\sim"{description}, from=10-5, to=11-6]
	\arrow["\sim"{description}, from=11-2, to=12-1]
	\arrow["\pi"{description}, from=11-4, to=10-3]
	\arrow["{\sh{(\epsilon_{Q'}\otimes \id)\oplus (\epsilon_{Q''}\otimes \id)} }"{description}, from=11-4, to=12-3]
	\arrow["\sim"{description}, from=11-4, to=14-6]
	\arrow["{\sh{\mathrm{proj}\otimes \id}}", from=11-6, to=14-6]
	\arrow["{+}"{description}, from=12-1, to=14-1]
	\arrow["\pi"{description}, from=12-3, to=11-2]
	\arrow["\sim"{description}, from=12-3, to=13-2]
	\arrow["\sim"{description}, from=12-3, to=14-3]
	\arrow["\pi"{description}, from=13-2, to=12-1]
	\arrow["{\sh{+}}"{description}, from=13-2, to=14-1]
	\arrow["{\sh{+\otimes \id}}", from=14-3, to=15-2]
	\arrow["{ \sh{(\epsilon_{Q'}\oplus \epsilon_{Q''})\otimes \id} }", from=14-6, to=14-3]
	\arrow["\sim"{description}, from=15-2, to=14-1]
\end{tikzcd}}
\caption{Commutative diagram for \cref{lem:sum_of_traces}}\label{fig:sum_of_traces}
\end{figure}

\begin{thm}\label{prop:add_traces_naive}
    For an exact sequence 
$(P',f')\to (P,f)\to (P'',f'')$
in $\End(R,S;M,N)$, 
$\tr(f)=\tr(f')+\tr( f'')$.
\end{thm}

\begin{proof} Let $\alpha \colon P\to P'\oplus P''$ be the isomorphism induced by the splitting as right $S$-modules.  
Then we have the following commutative diagram.

\begin{equation}\label{eq:split_on_the_ses}\begin{tikzcd}
	& {M\otimes (P'\oplus P'')} \\
	{M\otimes P'} & {M\otimes P} & {M\otimes P''} \\
	{P'\otimes N} & {P\otimes N} & {P''\otimes N} \\
	& {(P'\oplus P'')\otimes N}
	\arrow["{\text{id}\otimes\alpha^{-1}}", from=1-2, to=2-2]
	\arrow["{\text{id}\otimes i_{P'}}", from=2-1, to=1-2]
	\arrow[from=2-1, to=2-2]
	\arrow["{f'}"', from=2-1, to=3-1]
	\arrow[from=2-2, to=2-3]
	\arrow["f"', from=2-2, to=3-2]
	\arrow["{\text{id}\otimes i_{P''}}"', from=2-3, to=1-2]
	\arrow["{f''}"', from=2-3, to=3-3]
	\arrow[from=3-1, to=3-2]
	\arrow["{i_{P'}\otimes\text{id}}"', from=3-1, to=4-2]
	\arrow[from=3-2, to=3-3]
	\arrow["{\alpha\otimes \text{id}}", from=3-2, to=4-2]
	\arrow["{i_{P''}\otimes\text{id}}", from=3-3, to=4-2]
\end{tikzcd}
\end{equation}

By cyclicity of trace \cite[4.5.4]{ponto}, 
    \begin{align*}
    \tr(f)
    &=\tr(M\otimes P\xto{f}P\otimes N\xto{\alpha\otimes \id}(P'\oplus P'')\otimes N
    \xto{\alpha^{-1}\otimes \id}P\otimes N) 
\\ 
&=\tr(M\otimes (P'\oplus P'')
\xto{\id \otimes \alpha^{-1}}M\otimes P\xto{f}P\otimes N\xto{\alpha\otimes \id}(P'\oplus P'')\otimes N)
\end{align*}
By \cref{lem:sum_of_traces}, this is the sum of the traces of the maps 
\begin{equation}\label{eq:sum_of_traces_1}M\otimes P'\xto{\id \otimes i_{P'}} M\otimes (P'\oplus P'')
\xto{\id \otimes \alpha^{-1}}M\otimes P\xto{f}P\otimes N\xto{\alpha\otimes \id }(P'\oplus P'')\otimes N\xto{\pi_{P'}\otimes \id}P'\otimes N\end{equation}
and 
\begin{equation}\label{eq:sum_of_traces_2}M\otimes P''\xto{\id \otimes i_{P''}} M\otimes (P'\oplus P'')
\xto{\id \otimes \alpha^{-1}}M\otimes P\xto{f}P\otimes N\xto{\alpha\otimes \id }(P'\oplus P'')\otimes N\xto{\pi_{P''}\otimes \id}P''\otimes N\end{equation}
Using the commutative diagram in \eqref{eq:split_on_the_ses},
the composite in \eqref{eq:sum_of_traces_1} simplifies to 
\[M\otimes P'\xto{f'} P'\otimes N\xto{i_{P'}\otimes \id } (P'\oplus P'')\otimes N
\xto{\pi_{P'}\otimes \id}P'\otimes N,\] 
which equals  $f'$.  Similarly, the composite in
\eqref{eq:sum_of_traces_2} simplifies to 
\[M\otimes P''\xto{f''} P''\otimes N\xto{i_{P''}\otimes \id } (P'\oplus P'')\otimes N
\xto{\pi_{P''}\otimes \id}P''\otimes N\] 
which equals  $f''$. 
\end{proof}

\restatedefntraceequiv*

\bibliographystyle{amsalpha2}
\bibliography{trace}{}

\providecommand{\bysame}{\leavevmode\hbox to3em{\hrulefill}\thinspace}
\providecommand{\MR}{\relax\ifhmode\unskip\space\fi MR }
\providecommand{\MRhref}[2]{%
  \href{http://www.ams.org/mathscinet-getitem?mr=#1}{#2}
}
\providecommand{\doi}[1]{%
  doi:\href{https://dx.doi.org/#1}{#1}}
\providecommand{\arxiv}[1]{%
  arXiv:\href{https://arxiv.org/abs/#1}{#1}}
\providecommand{\href}[2]{#2}
\begin{thebibliography}{DKNP22}

\bibitem[Alm73]{ALM1}
G.~Almkvist, \emph{Endomorphisms of finitely generated projective modules over
  a commutative ring}, Ark. Mat. \textbf{11} (1973), 263--301.

\bibitem[Alm78]{ALM3}
\bysame, \emph{{$K$}-theory of endomorphisms}, J. Algebra \textbf{55} (1978),
  no.~2, 308--340.

\bibitem[Cam]{facets}
J.~A. Campbell, \emph{Facets of the {W}itt vectors}. \arxiv{1910.10206}

\bibitem[CP19]{morita}
J.~A. Campbell and K.~Ponto, \emph{Topological {H}ochschild homology and higher
  characteristics}, Algebr. Geom. Topol. \textbf{19} (2019), no.~2, 965--1017.
  \doi{10.2140/agt.2019.19.965} \arxiv{1803.01284}

\bibitem[DP78]{Dold1978-ol}
A.~Dold and D.~Puppe, \emph{Duality, trace, and transfer}, Proceedings of the
  International Conference on Geometric Topology, Warsaw; Warsaw, 1978,
  pp.~81--102.

\bibitem[DKNP22]{DKNP}
E.~Dotto, A.~Krause, T.~Nikolaus, and I.~Patchkoria, \emph{Witt vectors with
  coefficients and characteristic polynomials over non-commutative rings},
  Compos. Math. \textbf{158} (2022), no.~2, 366--408.
  \doi{https://doi.org/10.1112/s0010437x22007254} \arxiv{2002.01538}

\bibitem[DS88]{DRESS198887}
A.~W. Dress and C.~Siebeneicher, \emph{The {B}urnside ring of profinite groups
  and the {W}itt vector construction}, Advances in Mathematics \textbf{70}
  (1988), no.~1, 87--132. \doi{https://doi.org/10.1016/0001-8708(88)90052-7}

\bibitem[DS89]{DRESS19891}
A.~W. Dress and C.~Siebeneicher, \emph{The {B}urnside ring of the infinite
  cyclic group and its relations to the necklace algebra, {$\lambda$}-rings,
  and the universal ring of {W}itt vectors}, Advances in Mathematics
  \textbf{78} (1989), no.~1, 1--41.
  \doi{https://doi.org/10.1016/0001-8708(89)90027-3}

\bibitem[Gep18]{gepner}
D.~Gepner, \emph{The cyclotomic trace}, Arbeitsgemeinschaft: Topological Cyclic
  Homology (L.~Hesselholt and P.~Scholze, eds.), vol.~15, Oberwolfach Rep.,
  no.~2, 2018, pp.~874--879. \doi{https://doi.org/10.4171/OWR/2018/15}

\bibitem[Gra77]{GRAYSON2}
D.~R. Grayson, \emph{The {$K$}-theory of endomorphisms}, J. Algebra \textbf{48}
  (1977), no.~2, 439--446.

\bibitem[Gra78]{GRAYSON1}
\bysame, \emph{Grothendieck rings and {W}itt vectors}, Comm. Algebra \textbf{6}
  (1978), no.~3, 249--255.

\bibitem[Hes97]{Hesselholt1997WittVO}
L.~Hesselholt, \emph{Witt vectors of non-commutative rings and topological
  cyclic homology}, Acta Mathematica \textbf{178} (1997), 109--141.

\bibitem[KN18]{Lectures}
A.~Krause and T.~Nikolaus, \emph{Lectures on topological hochschild homology
  and cyclotomic spectra}, available at
  https://www.uni-muenster.de/IVV5WS/WebHop/user/nikolaus/Papers/Lectures.pdf,
  2018.

\bibitem[Lei]{leinster}
T.~Leinster, \emph{Basic bicategories}. \arxiv{math/9810017}

\bibitem[MP22]{mp:coherence}
C.~Malkiewich and K.~Ponto, \emph{Coherence for bicategories, lax functors, and
  shadows}, Theory and Applications of Categories \textbf{38} (2022), no.~12,
  328--373. \arxiv{2109.01249}

\bibitem[MS06]{may2004parametrized}
J.~P. May and J.~Sigurdsson, \emph{Parametrized homotopy theory}, Mathematical
  Surveys and Monographs, vol. 132, American Mathematical Society, Providence,
  RI, 2006. \doi{https://doi.org/10.1090/surv/132}

\bibitem[Pon10]{ponto}
K.~Ponto, \emph{Fixed point theory and trace for bicategories}, Ast\'{e}risque
  (2010), no.~333, xii+102. \arxiv{0807.1471}

\bibitem[PS12]{Ponto_2012}
K.~Ponto and M.~Shulman, \emph{Shadows and traces in bicategories}, Journal of
  Homotopy and Related Structures \textbf{8} (2012), no.~2, 151--200.
  \doi{https://doi.org/10.1007/s40062-012-0017-0} \arxiv{0910.1306}

\bibitem[PS16]{ps:linearity}
\bysame, \emph{The linearity of traces in monoidal categories and
  bicategories}, Theory Appl. Categ. \textbf{31} (2016), Paper No. 23,
  594--689. \arxiv{1406.7854}

\bibitem[Qui73]{Q72}
D.~Quillen, \emph{Higher algebraic {$K$}-theory. {I}}, Algebraic {$K$}-theory,
  {I}: {H}igher {$K$}-theories ({P}roc. {C}onf., {B}attelle {M}emorial {I}nst.,
  {S}eattle, {W}ash., 1972), Lecture Notes in Math., vol. Vol. 341, Springer,
  Berlin-New York, 1973, pp.~85--147.

\bibitem[{The}24]{stacks-project}
{The Stacks project authors}, \emph{The stacks project},
  \url{https://stacks.math.columbia.edu}, 2024.

\bibitem[Wei94]{weibel:homological}
C.~A. Weibel, \emph{An introduction to homological algebra}, Cambridge Studies
  in Advanced Mathematics, vol.~38, Cambridge University Press, Cambridge,
  1994. \doi{10.1017/CBO9781139644136}

\bibitem[Wei13]{weibel:kbook}
\bysame, \emph{The {$K$}-book}, Graduate Studies in Mathematics, vol. 145,
  American Mathematical Society, Providence, RI, 2013, An introduction to
  algebraic $K$-theory. \doi{10.1090/gsm/145}

\bibitem[Wit37]{Witt1937}
E.~Witt, \emph{Zyklische {K}\"orper und {A}lgebren der {C}harakteristik {$p$}
  vom {G}rad pn. {S}truktur diskret bewerteter perfekter {K}\"orper mit
  vollkommenem {R}estklassenk\"orper der charakteristik p.}, Journal für die
  reine und angewandte Mathematik \textbf{176} (1937), 126--140.

\end{thebibliography}

\end{document}